\documentclass{article}
\title{The $a$-number, $p$-rank and Cartier points of genus $4$ curves }

\author{Catalina Camacho-Navarro   \thanks{Universidad de Costa Rica anacatalina.camacho@ucr.ac.cr } \\
	\\ 
}

\date{}

\usepackage{amssymb}
\usepackage{amsthm} 
\usepackage{url}
\usepackage{amsmath}

\usepackage{float}
\usepackage{multirow}
\usepackage{multicol}

\newtheorem{thm}{Theorem}[section]
\newtheorem{prop}[thm]{Proposition}
\newtheorem{lem}[thm]{Lemma}
\newtheorem{cor}[thm]{Corollary}

\newtheorem{algo}[thm]{Algorithm}

\theoremstyle{definition}
\newtheorem{definition}[thm]{Definition}

\newtheorem{example}[thm]{Example}

\newtheorem{notation}[thm]{Notation}

\theoremstyle{remark}

\newcommand{\eref}[1]{\eqref{#1}}        % cite equation
\newcommand{\fref}[1]{Figure~\ref{#1}}   % cite figure
\newcommand{\cref}[1]{Chapter~\ref{#1}}  % cite chapter
\newcommand{\sref}[1]{Section~\ref{#1}}  % cite section/sub(sub)section
\newcommand{\aref}[1]{Appendix~\ref{#1}} % cite appendix
\newcommand{\tref}[1]{Table~\ref{#1}}    % cite table

\newcommand{\N}{\mathbb{N}}
\newcommand{\A}{\mathbb{A}}
\newcommand{\PP}{\mathbb{P}}
\newcommand{\F}{\mathbb{F}}
\newcommand{\Fbar}{\overline{\F}}
\newcommand{\kbar}{\overline{k}}
\newcommand{\Kbar}{\overline{K}}
\newcommand{\R}{\mathbb{R}}
\newcommand{\Z}{\mathbb{Z}}

\newcommand{\ds}{\displaystyle}
\newcommand{\calA}{\mathcal{A}}
\newcommand{\calF}{\mathcal{F}}
\newcommand{\calM}{\mathcal{M}}
\newcommand{\calC}{\mathcal{C}}
\newcommand{\calO}{\mathcal{O}}
\newcommand{\Car}{\mathscr{C}}
\newcommand{\HZOX}{H^0(X,\Omega_X^1)}
\newcommand{\HOOX}{H^1(X,\calO_X)}
\newcommand{\rk}{\text{rk}}
\newcommand{\rank}{\text{rank}}
\newcommand{\Hom}{\text{Hom}}

\newcommand{\Gal}{\text{Gal}}
\newcommand{\GL}{\text{GL}}
\newcommand{\Spec}{\text{Spec}}
\newcommand{\Div}{\text{Div}}
\newcommand{\dv}{\text{div]}}
\newcommand{\Fr}{\text{Fr}}
\newcommand{\Vr}{\text{Vr}}
\newcommand{\End}{\text{End}}
\newcommand{\cha}{\text{char}}
\newcommand{\Aut}{\text{Aut}}
\newcommand{\Zprof}{\hat{\Z}}
\newcommand{\Jac}{\text{Jac}}
\newcommand{\avec}{{\vec{a}}}
\newcommand{\calB}{\mathcal{B}}
\newcommand{\bv}{\mathbf{v}}
\newcommand{\bw}{\mathbf{w}}
\newcommand{\magma}{\texttt{Magma}}

\newcommand{\Ni}[1]{\mathbf{N1i}_{#1}}
\newcommand{\Nii}[1]{\mathbf{N1ii}_{#1}}
\newcommand{\Ntwo}[1]{\mathbf{N2}_{#1}}
\newcommand{\Ds}[1]{\mathbf{D}_{#1}}

\newcommand{\defi}[1]{\textsf{#1}} % for defined terms

\usepackage[numbers, square, comma, sort&compress]{natbib}
\usepackage{hyperref}
\begin{document}

	\maketitle
	
	\begin{abstract}
		We study genus $4$ curves over finite fields and two invariants of the $p$-torsion part of their Jacobians: the $a$-number ($a$) and $p$-rank ($f$). We collect and analyze statistical data of curves over $\F_p$ for $p=3,5,7,11$ and their invariants. Then, we study the existence of Cartier points, which are also related to the structure of $J[p]$. For curves with $0\leq a<g$, the number of Cartier points is bounded, and it depends on $a$ and $f$.
		
	\end{abstract}
	\textbf{Keywords:} cartier points, genus 4, a-number, p-rank, Hasse-Witt Matrix.
	
	%\tableofcontents
	\section{Introduction}\label{section:intro}
	Let $X$ be a smooth projective genus $g$ curve over a field $k$ of characteristic $p$. The Torelli map associates $X$ with its Jacobian $J_X$, a principally polarized abelian variety of dimension $g$. The map embeds the moduli space $\calM_g$ of curves of genus $g$ into $\calA_g$, the moduli space of principally polarized abelian varieties of dimension $g$ over $k$. In consequence, it allows us to study the stratification of $\calM_g$ by looking at the group scheme structure of $J_X[p]$, the $p$-torsion part of the Jacobian. This is called the Ekedahl--Oort stratification. 
	For $g=2,3$, the Torelli locus is open and dense in $\calA_g$. For $p\ge 3$ and $g\leq 3$ this can be used to show that  all Ekedahl--Oort types occur for the Jacobians of smooth curves $X/\overline{\F}_p$ (\cite{Pries_sh}). The same is not known for $g\geq 4$. 
	
	Motivated by this and other similar open questions related to the $p$-torsion part of the Jacobian, we study smooth irreducible curves with $g=4$. We focus on the non-hyperelliptic kind. In particular, we look at the $a$-number and the $p$-rank, which are two invariants of $J_X[p]$. 
	
	In order to obtain a database of smooth, irreducible, genus $4$ non-hyperelliptic curves, we restrict our analysis to what we define as curves in \textit{standard form}. Recall that if $X$ is a curve with the above properties, then it has a model given by the zero locus of a quadratic and a cubic homogeneous polynomials in $k[x,y,z,w]$. Kudo and Harashita show in \cite{Kudo&Harashita} that under some assumptions, the defining equations can be simplified to reduce the number of cases. The curves given by these simplified equations are \textit{curves in standard form}. 
	
	We gather a statistical sample of curves in standard form defined over $\F_p$ for $p\in \{3,5,7,11\}$. For each of them we find the Hasse--Witt matrix $H$ and use it to compute the $a$-number and $p$-rank: the $a$-number is $g-\rank(H)$ and the $p$-rank is $f=\rank \left(HH^{(p)}\cdots H^{\left(p^{g-1}\right)}\right)$.  As one should expect, the majority of curves appear in the sample are ordinary, and the percentages decrease as the $a$-number increases (or similarly, as the $p$-rank decreases).
	
	We also explore in this paper the concept of Cartier point. We say that $P\in X(\overline{k})$ is a Cartier point if the hyperplane of regular differentials of $X$ vanishing at $P$ is stable under the Cartier operator. Baker introduces the definition in \cite{Baker} and remarks that they are related to the $p$-torsion points of the Jacobian.
	
	If $X$ has $a$-number $0\leq a<g$ then there is an upper bound on the number of Cartier points of $X$ given by Baker \cite{Baker}. When $a\ne 0$,  we classify Cartier points in Type 1 and Type 2 (see Definition~\ref{def:T1andT2}). The maximum number of Type 1 points depends on the $a$-number and the maximum number of Type 2 points depends on the $p$-rank. We are interested in determining the conditions under which these bounds are attained when $X$ is non-ordinary. Therefore, we develop algorithms to find all of the Cartier points on curves in standard form and apply them to our database. 
	
	The Cartier points are particularly interesting when $a=g-1$, because we can assign multiplicity to each of them. This is why we later focus on curves with $a=3$. We find all of the curves in standard form with $a=3$, defined over $\F_p$ for $p=3,5$ and a subset of them over $\F_7$. We explore the possible degrees and multiplicity distributions of these points.

	Here are some of the most relevant conclusions from our work, concerning non-hyperelliptic genus $4$ curves in standard form:
	
	\begin{enumerate}
		\item In our smooth sample, the are no curves with $(a,f)=(1,0)$ over $\F_p$ for $p\in \{3,5,7,11\}$. (Corollary~\ref{cor:no_a1_f0_curves}).
		\item There are, up to $\F_3$-isomorphism, exactly 27 curves with $a$-number $3$ over $\F_3$ in standard form. All of them have $p$-rank $1$. (Corollaries~\ref{cor:exhaustive_search_p3} and~\ref{cor:no_prank0_p3_curves}).
		\item There are, up to $\F_5$-isomorphism, exactly $134$ curves with  $a=3$ over $\F_5$. (Corollary~\ref{cor:exhaustive_search_p5}).

		\item 	In our smooth sample, no curve with $a$-number $2$ reaches the bound of $6$ Type $1$ Cartier points. Moreover, the maximum number of Type $1$ points attained on curves with $a$-number $2$ is $3$ for $p\in \{5,7,11\}$ and $2$ for $p=3$. (Corollary~\ref{cor:no_a2_curves_reach_UB_of_T1}).

		\item	In our smooth sample no curve with $p$-rank $2$ or $3$ reaches the bound of $6$ Type $2$ Cartier points. The maximum number of points that occurs is $3$ and $4$, respectively. (Corollary~\ref{cor:no_f2_curves_reach_UB_of_T2}).

		\item When $a=3$, the bound on Type 1 points is sharp for $p\in \{5,7,11\}$ and the total bound for both Types is sharp for $p=7$. (Corollaries~\ref{cor:summary_T1_bounds} and~\ref{cor:bounds_reached_p7}).
		\item There are no curves in standard form over $\F_3$ with $a=3$ that attain either of the upper bounds for Cartier points. (Lemma~\ref{lem:T1_points_in_a3_p3_curves}).
		
	\end{enumerate}
	
	\section{Acknowledgements}
	This paper is part of my doctoral dissertation. 
	The work presented here would not have been possible without the support of my advisor Rachel Pries. I would like to thank Dr. Pries for suggesting this topic and sharing her knowledge throughout my time in graduate school. 
	I was financially supported by University of Costa Rica while I worked on this project.
	
	\section{Preliminaries}\label{sec:preliminaries}
	This section includes the background information related to the Cartier operator, Hasse--Witt matrix, $p$-rank and $a$-number of a curve. Unless otherwise stated, $p$ will be an odd prime number and $k$ a perfect field of characteristic $p$. 
	
	\subsection{The $p$-rank  }\label{subsec:anumber_prank}
	
	Let $X$ be a smooth irreducible genus $g$ curve over $k$. 
	Denote by $J_X$ the Jacobian variety associated to $X$. Then $J_X$ is an abelian variety of dimension $g$  isomorphic to $\text{Pic}^0_{X/k}$, equipped with a principal polarization. 
	
	The $p$-torsion part of the Jacobian, denoted by $J_X[p]$, is a group scheme and here we will study two invariants associated to it.  The first one is the $p$-rank, defined as the integer $f$ such that $\# J_X[p](k)=p^f$. The $p$-torsion subgroup of an abelian variety of dimension $g$ has order at most $g$, hence $0\leq f\leq g$.

	\subsection{Cartier and Frobenius operators and the $a$-number}\label{subsec:CartierFrob}

	Suppose that $x$ is a separating variable of $k(X)/k$, then every $t \in k(X)$ can be written as
	\begin{align}\label{eq:f_inK(X)}
	t=t_0^p+t_1^px+\ldots +t_{p-1}^px^{p-1},
	\end{align}
	with $t_i \in k(X)$.
	\begin{definition}\label{def:cartier_operator}
		The Cartier operator $\calC$ is defined on $\Omega_X^1$ for $t$ as above by
		\begin{align}
		\calC(tdx)=t_{p-1}dx.
		\end{align}
		
	\end{definition}
	The Cartier operator is $1/p$-linear, meaning that $\calC(a^p\omega_1+b^p\omega_2)=a\calC(\omega_1)+b\calC(\omega_2)$, for $a,b$ in $k(X)$ and $\omega_1, \omega_2 \in \Omega_X^1$.  It induces a well defined map $\calC$ on the $k$-vector space of regular differentials $H^0(X,\Omega_X^1)$. 

	\begin{definition}\label{def:CartierMatrix}
		If $B=\left\lbrace\omega_1, \ldots, \omega_g\right\rbrace$ is a $k$-basis for $H^0(X,\Omega_X^1)$ and $\calC(\omega_j)=\sum_{i=1}^g{c_{ij}\omega_i}$ then the Cartier--Manin matrix of $X$ with respect to this basis is the matrix $(c_{ij}^p)_{ij}$. 
	\end{definition}

	\begin{definition}\label{def:frobenius_in H1}
		The absolute Frobenius of $X$ is the morphism $\calF:X\to X$ given by the identity on the underlaying topological space and $t\mapsto t^p$ on $\calO_X$. Let $\calF_X$ be the induced endomorphism in $H^1(X,\calO_X)$. We call $\calF_X$ the Frobenius endomorphism.
	\end{definition}
	The Frobenius endomorphism is $p$-linear, that is $\calF_X(a\xi)=a^p\calF_X(\xi)$ for all $a\in k$ and all $\xi \in H^1(X,\calO_X)$. 
	
	\begin{definition}\label{def:HW_matrix}
		Let $B'=\{\xi_1, \ldots, \xi_g\}$ be a $k$-basis of $H^1(X,\calO_X)$ and $\calF_X(\xi_j)=\sum_{i=1}^{g}a_{ij}\xi_i$ for some $a_{ij}\in k$. Then the Hasse--Witt matrix of $X$ with respect to $B'$ is the matrix $(a_{ij})_{ij}$.	
	\end{definition}

	The space $H^1(X,\calO_X)$ is the dual of $H^0(X,\Omega_X^1)$  and there is a perfect paring $<,>$ on $H^1(X,\calO_X)\times H^0(X,\Omega_X^1)$ such that $<\calF_X\xi,\omega>=<\xi,\calC\omega>^p$.
	
	When $B$ and $B'$ are dual basis,the Cartier--Manin matrix $M$ with respect to $B$ is the transpose of the Hasse--Witt matrix $H$ with respect to $B'$.

	By Serre \cite{serre1972}, the $p$-rank $f$ is the stable rank of the Frobenius and since this operator is $p$-linear, it implies that
	\begin{equation}\label{eq:p-rank-stablerank}
	f=\rank \left(HH^{(p)}\cdots H^{\left(p^{g-1}\right)}\right),
	\end{equation}
	where $H^{(i)}$ is the matrix obtained by raising every entry of $H$ to the $i$-th power. 
	
	The second invariant of $J_X[p]$ that we study here is the $a$-number, defined as the rank of the Cartier operator (see Oort \cite{oort_li}), i.e.,
	\begin{equation}\label{eq:a_number_rank}
	a=g-\rank(H)=g-\rank(M).
	\end{equation}
	It is known that $0\leq a+f \leq g$. Generically $f=g$, in which case $X$ is said to be \textit{ordinary}. 
	The other extreme case is when $X$ is \textit{superspecial} and it occurs when $a=g$ or equivalently, when the Cartier operator is identically $0$ on $H^0(X,\Omega_X^1)$.

	In Section~\ref{subsec:HWmatrix_genus4} we follow the work from \cite{Kudo&Harashita} to describe how to explicitly find the Hasse--Witt matrix, when $X$ is the complete intersection of two homogeneous polynomials and use it to compute the $p$-rank and $a$-number.

	\subsection{Previous results}\label{subsec:previous_results}
	In this section we review some of the main and more recent results with respect to genus $g$ curves of positive characteristic and possible values for the $a$-number and $p$-rank that occur. 
	
	\begin{thm}[Ekedahl \cite{Ekedahl:Supersingular},Theorem 1.1]\label{thm:Ekedahl}
		Let $X$ be a smooth curve of genus $g$ over an algebraically closed field $k$ of characteristic $p>0$. If $X$ is superspecial then 
		\begin{enumerate}
			\item $g\leq \frac{1}{2}(p^2-p)$ and
			\item $g\leq \frac{1}{2}(p-1) $ if $X$ is hyperelliptic and $(p,g)\neq (2,1).$
		\end{enumerate}
	\end{thm}

	Baker gives in \cite{Baker} an alternative proof for Theorem~\ref{thm:Ekedahl}, based on the existence of Cartier points (see Section~\ref{sec:CartierPoints} for definition). Let $m$ be the rank of the Cartier operator. Re (\cite{RE}, Theorem 3.1 and Proposition 3.1) provides a generalization of this result to any value of $m$. The author proves that in fact $ g\leq (m+1)p\frac{(p-1)}{2}+pm$ and if $X$ is also hyperelliptic then $g<\frac{p+1}{2}+mp$.

	Zhou (\cite{Zhou}, Theorem 1.1) gives a strengthening of this result for the case when $m=1$ that gives a bound for $g$ of $p+\frac{p(p-1)}{2} $. Moreover, Frei (\cite{Frei18}, Theorem 3.1]) proved that if $g\geq p$ where $p$ is an odd prime, then there are no smooth hyperelliptic curves of genus $g$ defined over a field of characteristic $p$ with $a$-number equal to $g - 1$.

	There are examples, however, of curves with $a$-number $g-1$ that are non-hyperelliptic. Also Zhou finds in \cite{Zhou_anumber} a family of Artin--Schreier curves with these properties. 
	
	More generally, Pries \cite{PriesPtorsion09} proves the existence of smooth curves with $a$-number $1,2$ and $3$, under certain conditions.

	\section{Genus $4$ non-hyperelliptic curves}\label{sec:genus4}
	
	There are currently many open questions concerning the existence of curves with certain $p$-ranks and $a$-numbers, given a fixed genus $g$. For instance, there exist curves of genus $2$ and 3 with any possible $p$-rank and $a$-number over fields of characteristic $p$, with the exception of superspecial curves of genus $2$ when $p=2$ and superspecial curves of genus $3$ when $p=2,3$. For $g\geq 4$, however, it is not known if this happens. For example, consider Question 3.6 in \cite{PriesCRN}: \textit{For all $p$, does there exist a smooth curve of genus $4$ with $p$-rank $0$ and $a$-number at least $2$?}
	
	One can find in the literature partial answers to the last and similar questions. For example, if $p=3$ then by Ekedahl's Theorem (\ref{thm:Ekedahl}), there is no curve of genus $4$, with  $a$-number 4. In \cite{Zhou_EOg4}, Zhou (building on work from \cite{Frei18}, \cite{GlassPries} and \cite{PriesPtorsion09}) shows, by studying a family of Artin--Schreier curves that in characteristic 3, there are genus $4$ curves with $a$-number $a$ and $p$-rank $f$  for all $a\leq 2$ and $ f\leq a$ and for $(a,f)=(3,1)$.

	In fact, the author computes the Ekedahl-Oort types to show that the corresponding locus of $\calM_g$ is non empty of codimension at most 6. In Section~\ref{sec:CP_on_GC_p=3} we provide additional examples of genus $4$ curves with $a$-number $3$ and $p$-rank $1$ over $\F_{3}$. We know by Frei ~\cite{Frei18} that there are no $a$-number $3$ and $p$-rank $0$ genus $4$ hyperelliptic curves, so one can ask whether it is possible to have a non-hyperelliptic genus $4$ with those invariants.

	Kudo and Harashita \cite{Kudo&Harashita} also studied genus $4$ curves, they prove two results related to non-hyperelliptic superspecial curves of genus $4$. In particular they show that there are no superspecial genus $4$ curves in characteristic $7$, and that over $\F_{25}$ they are all isomorphic to  
	$$2yw+z^2=0,  x^3+a_1y^3+a_2w^3+a_3zw^2=0, $$
	in $\mathbb{P}_{\F_{25}}^3$, where $a_1,a_2\in \mathbb{F}^\times_{25}$ and $a_3\in  \mathbb{F}_{25}.$

	\subsection{Defining equations of genus $4$ non-hyperelliptic curves}\label{subsec:equations_genus4}

	If $X$ is a genus $4$, smooth, irreducible and non-hyperelliptic curve, then the canonical map embeds $X$ into $\PP^3_k$ as the intersection of the zero loci of a quadratic and a cubic homogeneous polynomial in four variables (see \cite{Hartshorne77}), that we detail below.

	We restrict our computations to what we will define as genus $4$ curves in \textit{standard form}. These are based on equations given by Kudo and Harashita \cite{Kudo&Harashita}.

	\subsubsection{Quadratic forms and reduction of cubics}
	As explained in \cite{Kudo&Harashita}, any irreducible quadratic form in $k[x,y,z,w]$ is equivalent to one of $ F_1=2xw+2yz$, $F_2=2xw+y^2-\epsilon z^2$ or $F_d=2yw+ z^2$ with some $\epsilon\notin (k^\times)^2$. Therefore we can assume that $X$ has a model given by $V(F,G)$, with $F $ being one of $F_1$, $F_2$ or $F_d$, and $G$ a homogeneous polynomial of degree $3$. The  cubic $G$ can be reduced by changes of variables, induced by the action of the orthogonal similitude groups associated to the quadratic forms. This is done in detail in Section 4 of \cite{Kudo&Harashita}. The simplified equations provide the following definition.

	\begin{definition}\label{def:standardform}
		Let $ F_1=2xw+2yz$, $F_2=2xw+y^2-\epsilon z^2$ or $F_d=2yw+ z^2$ with $\epsilon\notin (k^\times)^2$. We say that a curve $X$ of genus $4$  over $k$  is in \textit{standard form} if it is non-hyperelliptic, irreducible, smooth and $X=V(F,G)$ with \\

		\textbf{(Case D)} $F=F_d$ and
		\begin{align*}
		G& =a_0x^3+(a_1y^2+a_2z^2+a_3w^2+a_4yz+a_5zw)x+a_6y^3+a_7z^3+a_8w^3+a_9yz^2\\
		&+ b_1z^2w+b_2zw^2,
		\end{align*}\label{eq:FandG2_p=5}
		for $a_i \in k$ and $a_0,a_6\in k^\times$, with $b_1,  b_2 \in \{0,1\}$ and the leading coefficient of $r=a_1y^2+a_2z^2+a_3w^2+a_4yz+a_5zw$ is $1$ or $r=0$; or\\

		\textbf{(Case N1i)} $F=F_1$ and 
		\begin{align*}
		G&=(a_1y +a_2z)x^2 +a_3yzx+y^3 +a_4z^3 +b_1y^2z +a_5yz^2 \\ &+(a_6y^2 +a_7yz +b_2z^2)w+(a_8y +a_9z)w^2 +a_{10}w^3,
		\end{align*}
		for $a_i \in k $ with $a_1\neq 0$, $a_2\neq 0$ and for $b_1\in \{0\} \cup k^{\times}/(k^{\times})^2$ and $b_2 \in \{0,1\}$; or \\
		
		\textbf{(Case N1ii)} $F=F_1$ and 
		\begin{align*}
		G=&	(a_1y +a_2z)x^2 +a_3yzx+b_1y^2z +b_2yz^2 +(a_4y^2 +a_5yz +b_3z^2)w+(a_6y +\\ &a_7z)w^2 +a_8w^3,
		\end{align*}
		for $a_i \in k $ with $a_1a_2\neq 0$ and for $b_1,b_3 \in \{0,1\}$ and $b_2 \in \{0\} \cup k^{\times}/(k^{\times})^2$; or\\

		\textbf{(Case N2)} $F=F_2$ and 
		\begin{align*}
		\begin{split}
		G=&(a_1y + a_2z)x^2 + a_3(y^2 - \epsilon z^2)x + b_1y(y^2 -\epsilon z^2) + a_4y(y^2 + 3\epsilon z^2) +  a_5z(3y^2 \\&+\epsilon z^2)+(a_6y^2 + a_7yz + b_2z^2)w + (a_8y + a_9z)w^2 + a_{10}w^3,
		\end{split}
		\end{align*}
		for $a_i \in k$, with $(a_1,a_2)\neq (0,0)$ and $b_1,b_2 \in \{0,1\}$ and $\epsilon$ a non-trivial fixed representative of $k^\times/(k^\times)^2$,

	\end{definition}

	We remark that every genus $4$ non-hyperelliptic can be written as\linebreak  $X=V(F,G)$, where $F$ is one of $F_1, F_2, F_d$. Moreover,  $G$ can be simplified to standard form if the curve has least 37 points. If the $F=F_d$ then it is enough that $\#k>5$ (\cite{Kudo&Harashita} Lemmas 4.3.1, 4.4.1 and 4.5.1).

	\subsection{Hasse--Witt Matrix of genus $4$ non-hyperelliptic curves}\label{subsec:HWmatrix_genus4}
	Let $X=V(F,G)$ be the complete intersection on $\PP^3_k$ defined by homogeneous polynomials $F$ and $G$ in $k[x,y,z,w]$ of degrees $d$ and $c$, respectively.  Following \cite{Hartshorne77} and \cite{Baker} we see that $H^1(X,\mathcal{O}_X)\cong H^3(\PP^3,\mathcal{O}_\PP^3(-c-d))$, where the basis $\calB$ of $H^1(X,\mathcal{O}_X)$ that corresponds to the coordinates $x,y,z,w$ is associated to the basis of $ H^3(\PP^3,\mathcal{O}_\PP^3(-c-d))$ given by 
	\begin{align}\label{eq:basis_h3}
	\{x^ky^lz^mw^n : (k,l,m,n)\in (\Z_{<0})^4 \text{ and } -k-l-m-n=c+d \}.
	\end{align}
	Using this fact, Kudo and Harashita present an algorithm to compute the Hasse--Witt of such curves, which can be generalized to compute the corresponding matrix for any complete intersection over a perfect field of positive characteristic.
	
	\begin{prop}[Kudo and Harashita \cite{Kudo&Harashita}, Proposition 3.1.4]\label{prop:computation_hw}
		Let $X$ of genus $g$ be defined as above and suppose $(FG)^{p-1}=\sum c_{i_1,i_2,i_3,i_4}x^{i_1}y^{i_2}z^{i_3}w^{i_4}$. Then the Hasse--Witt matrix of $X$ is given by
		\begin{eqnarray}\label{eq:HWmatrix:general}
		\left(c_{-k_jp+k_i,-l_jp+l_i,-m_jp+m_1i-n_jp+n_i}\right)_{i,j}.
		\end{eqnarray}

	\end{prop}
	
	This formula gives a matrix $H$ that represents the action of $\calF_X$ by left matrix multiplication. Let $\bv$ be the column vector corresponding to an element of $H^1(X,\calO_X)$ expressed in terms of the basis $\calB$. Then the image of $\bv$ under $\mathcal{F}_X$ is given by $H\cdot \bv^{(p)}$, since  $\mathcal{F}_X$ is $p$-linear.

	In the next section, we use Proposition~\ref{prop:computation_hw}  together with  the equations from Section~\ref{sec:genus4} to compute examples of non-hyperelliptic smooth curves of genus $4$ with $a$-number $3$.

	\section{A database of curves in standard form over $\F_p$ }\label{sec:examples_goodcurves_fp}

	In this section we construct a database of genus 4 curves in standard form (Definition~\ref{def:standardform}) over $\F_p$ for $p \in \{3,5,7,11\}$. We restrict our data collection to non-ordinary and non-superspecial curves, that is, curves with $a$-number equal to $1,2$ or $3$. First let us explain the notation used:
	\begin{itemize}
		\item Let $\avec \in \F_p^{20}$. We denote by $G_{\avec}$ the cubic whose coefficients correspond to the entries of $\avec$, assigned to the monomials of degree 3 in $\F_p[x,y,z,w]$ in graded lexicographic order, which means that we consider $x>y>z>w$ and to order the monomials we first compare the exponents of $x$, then those of $y$ and so on. 
		
		\item For each one of the cases D, N1i, N1ii and N2 in Definition~\ref{def:standardform} we define a subset of $\F_p^{20}$ such that for all $\avec$ in that subset, the cubic $G_{\avec}$ has conditions. We denote each subset by $\Ds{p},\Ni{p},\Nii{p}$ and $\Ntwo{p}$.
		
	\end{itemize}
	As a reference, the cardinality of each of the sets above is shown in Table~\ref{tab:card_vectors_sets}.

	\begin{table}[h]  
		\caption{Cardinality of the sets $\Ds{p}$, $\Ni{p}$, $\Nii{p}$, $\Ntwo{p}$.}\label{tab:card_vectors_sets}
		\begin{center}
			\begin{tabular}{cc}
				Set	&	Cardinality	\\ \hline \hline
				$\Ds{p}$	&	$4p^3(p-1)(p^5+p-2)$	\\ 
				$\Ni{p}$	&	$6p^8(p-1)^2$	\\ 
				$\Nii{p}$	&	$12p^6(p-1)^2$	\\ 
				$\Ntwo{p}$	&	$4p^8(p^2-1)$	\\ \hline
			\end{tabular}
			
		\end{center}
	\end{table}
	To construct a curve in standard form we select an element $\avec$ in one of the subsets above together with the corresponding $F \in \{F_1,F_2,F_d\}$ and verify if $V(F,G_\avec)$ is non-singular, irreducible and  has genus $4$. In practice, we implement the algorithm in \texttt{Magma}, where we check each of the conditions as follows: 
	\begin{itemize}
		\item Irreducibility: using the intrinsic \texttt{Magma} function \texttt{IsIrreducible}, which verifies the condition over the base field by a Gr\"{o}bner basis computation.
		
		\item Genus: the command \texttt{Genus} finds the arithmetic genus of the projective normalization of the curve. 
		\item Nonsingularity: Implementing Algorithm \texttt{DetermineNonSingularity} (Algorithm 3.2.1 in \cite{Kudo&Harashita}), which is based on solving a radical membership problem on the minors of the Jacobian matrix of $V(F,G_\avec)$. The later is done by a \texttt{RadicalMembership} algorithm (\cite{Kudo&Harashita}, Appendix A) 
		
	\end{itemize}

	Next, suppose that $X=V(F,G_\avec)$ is a curve in standard form, then we need to determine its $a$-number and $p$-rank. We do this by computing the Hasse--Witt matrix $H$ of $X$  with respect to the basis of $\HOOX$ given by (\ref{eq:basis_h3}), as in Proposition~\ref{prop:computation_hw}. The $a$-number is equal to $4-\rank(H)$ and the $p$-rank $f$ is the rank of  $HH^{(p)}\cdots H^{\left(p^{g-1}\right)}$. In our case, since $F$ and $G_\avec$ are defined over $\F_p$, then $f=\rank(H^{g})$.

	\subsection{Collecting the data}\label{sec:data_collection_curves}

	We apply the above procedure to random samples of tuples in $\Ds{p}$, $\Ni{p}$, $\Nii{p}$ and $\Ntwo{p}$ for $p\in \{3,5,7,11\}$ in order to gather statistical information. We construct the sample using the intrinsic \texttt{Random} function in \texttt{Magma}. In addition, we classify the curves by $a$-number and $p$-rank. 
	
	When we analyze the samples of curves in our data, we will often compare the number of curves with certain $p$-rank and/or $a$-number with the total of smooth curves in standard form obtained. In our search we classify the curves with $a$-number $0,1,2$ and $3$. We ignore the curves with $a=4$ because they represent a very small proportion of the curves  and hence they do not affect the percentages in a significant way. 
	
	Also, our focus will be studying the occurrence of Cartier points (see Section~\ref{sec:CartierPoints}), and we already know that superspecial curves have infinitely many of them (Baker \cite{Baker}). Here is how we will refer to our different samples:
	\begin{itemize}
		\item \textbf{Total set}: the set $\Ds{p}$, $\Ni{p}$, $\Nii{p}$ or $\Ntwo{p}$, depending on the case.
		\item\textbf{Sampled set}: the curves of the sets above that are included in our random search.
		\item \textbf{Smooth sample}: the set of cubics from the sampled set that give smooth, irreducible, genus $4$ curves (excluding superspecial curves).
		\item \textbf{Singular sample}: the set of cubics from the sampled set that give curves that fail either one of the conditions for smoothness, irreducibility or genus. 
	\end{itemize}
	If we do not specify the subcase, then	\textbf{ Sampled set}, \textbf{Smooth sample} and \textbf{Singular sample} will refer to the total samples throughout the four cases. 
	
	\subsection{Summary of results of sampling search}\label{sec:summary_sampling}
	In this section we display the overall results from the sampling search. The sizes of our final Sampled sets by case are shown in \tref{tab:total_sample_sizes}.
	
	\begin{table}[h]
		\caption{Total sampled sizes over $\F_p$.}\label{tab:total_sample_sizes}
		\begin{center}
			\begin{tabular}{c|cccc|c}
				$p$	&	D	&	N1i	&	N1ii	&	N2	&	Total sample	\\	\hline \hline
				$3$	&	52704	&	92123	&	34992	&	6447	&	186266	\\	
				$5$	&	179728	&	179970	&	179434	&	179970	&	719102	\\	
				$7$	&	215957	&	225193	&	206890	&	214998	&	863038	\\	
				$11$	&	89999	&	91100	&	89999	&	90000	&	361098	\\	\hline
			\end{tabular}
		\end{center}
	\end{table}
	
	In Table~\ref{tab:summary_allcurves} we list the number of curves with $p$-ranks $f=0,1,2,3$ and $a$-number $a=1,2,3$, over $\F_p$ for $p\in \{3,5,7,11\}$. We include the totals for the singular samples and the ordinary curves (that is, those with $a$-number $0$). From this classification we have the following statement:
	
	\begin{cor}\label{cor:no_a1_f0_curves}
		In our data set of genus $4$ curves in standard form there are no curves with $p$-rank $0$ and $a$-number $1$, when $p\in \{3,5,7,11\}$.
	\end{cor}
	\begin{proof}
		See Table~\ref{tab:summary_allcurves}.
	\end{proof}

	\begin{table}[h]
		\caption{ Sample of curves in standard over $\F_p$.}\label{tab:summary_allcurves}
		\begin{center}
			\begin{tabular}{cccccc}
				$a$	&	$f$	&	3	&	5	&	7	&	11	\\	\hline\hline
				\multicolumn{2}{c}{Sampled set}			&	186266	&	719102	&	863038	&	361098	\\	\hline
				\multicolumn{2}{c}{Singular sample}			&	92654	&	251584	&	191925	&	81845	\\	\hline
				0	&		&	56983	&	370476	&	529394	&	253627	\\	\hline
				1	&	0	&	0	&	0	&	0	&	0	\\	
				1	&	1	&	1679	&	3592	&	1652	&	217	\\	
				1	&	2	&	4134	&	14615	&	10687	&	2146	\\	
				1	&	3	&	23485	&	74585	&	76142	&	23044	\\	
				&	Total	&	29298	&	92792	&	88481	&	25407	\\	\hline
				2	&	0	&	1157	&	183	&	44	&	3	\\	
				2	&	1	&	2095	&	790	&	231	&	21	\\	
				2	&	2	&	3379	&	3231	&	1624	&	194	\\	
				&	Total	&	6631	&	4204	&	1899	&	218	\\	\hline
				3	&	0	&	0	&	10	&	0	&	0	\\	
				3	&	1	&	700	&	36	&	5	&	1	\\	
				&	Total	&	700	&	46	&	5	&	1	\\	\hline
				\multicolumn{2}{c}{Smooth sample}			&	93612	&	467518	&	619779	&	279253	\\	\hline
			\end{tabular}
		\end{center}
	\end{table}

	Next we discuss the results from the search for each $p$. We include the sizes of the total set, sampled set, smooth sample and singular sample, normalized by $\log_{p}$. Then we show the break down of curves with $a=1,2,3$ and specify the percentage of the sampled set and smooth sample they represent.

	\subsubsection{Case $p=3$ .}\label{sec:goodcurvesp=3}
	
	We checked a total of $186266$ tuples, which corresponds to approximately $71\%$ of all the possible tuples. We saw that $50.26\%$ of them gave smooth curves and $19.66\%$ of the total had $a$-number $a=1,2,3$. %Table~\ref{tab:total_samples_p3} contains the sizes of the total set and the samples obtained from the search, normalized by $\log_3$. 
	Notice that we were able to sort all of the cubics from the total sets $\Ds{3}$ and $\Nii{3}$.

	\subsubsection{Case $p=5$. }\label{sec:goodcurvesp5}

	We selected a random sample of $719102$ tuples in $\Ds{5}, \Ni{5}, \Nii{5}$ and $\Ntwo{5}$. This is around $1.49\%$ of the total set. A $61.01\%$ of that sample corresponds to the Smooth sample, and $13.49\%$ of the total are non-ordinary curves.

	\subsubsection{Case $p=7$. }\label{sec:goodcurvesp=7}
	
	We analyzed a random sample of $863038$ pairs of tuples in $\Ds{7}$, $\Ni{7}$, $\Nii{7}$ and $\Ntwo{7}$, which corresponds to $0.06\%$ of the total. The smooth sample from this set has $619779$ tuples and $90385$ of them are curves in standard form with $a$-number $1,2$ or $3$. These correspond to $71.81\%$ and $10.47\%$ of the total, respectively. %

	\subsubsection{Case $p=11$}\label{sec:goodcurves_p11}
	
	In this case we did a random search that included $361098$ tuples in $\Ds{11}, \Ni{11}$, $ \Nii{11}$ and $\Ntwo{11}$, this is approximately $0.0002\%$ of the total set. Of this sample, $77.33\%$  are smooth and  $7.10\%$ of the total are non-ordinary.

	\subsection{The case $a=3$}\label{sec:summary_exhaustive}
	In order to focus on the analysis of some aspects of genus $4$ curves with $a$-number $3$, we apply the procedure above to all the tuples in $\Ds{p}, \Ni{p}, \Nii{p}$ and $\Ntwo{p}$, but only store the smooth, irreducible curves with $a=3$. We did this for $p=3,5$. The search is also done for $p=7$ but only for tuples in $\Ds{7}$ and a subset of $\Ni{7}$, because of the long computing times. It is important to remark that after the search, we classify the curves by $\F_p$-isomorphism classes.
	
	In the cases $p=3,5$ we also have a complete list of all curves in standard form that have $a$-number $3$. For $p=7$ we have a subset of them: the ones in case D and those in case N1i where $(b_1,b_2)=(0,0)$, we denote these lists by 7(D) and 7(N1i'). The information on these curves, classified by isomorphism classes, is shown in Table~\ref{tab:isom_classes_p3_a3}. A first conclusion that we can draw from this data is the following: 
	
	\begin{cor}\label{cor:no_prank0_p3_curves}
		There are no genus $4$ curves in standard form over $\F_3$ with $p$-rank $0$ and $a$-number $3$.
	\end{cor}
	
	\begin{table}[h]  
		\caption{Isomorphism classes of curves in standard form with $a=3$ over $\F_p$. }\label{tab:isom_classes_p3_a3}
		
		\begin{center}
			\begin{tabular}{ccccc}
				$p$-rank	&	3	&	5	&	7(D) &7(N1i')	\\	\hline \hline
				0	&	0	&	36	&	9&2	\\	
				1	&	27	&	98	&	56&	27\\	
				\hline
				Total	&	27	&	134	&	65&	29\\	\hline
			\end{tabular}
		\end{center}
	\end{table}
	
	\subsubsection{Case $p=3$}
	
	There are a total of $1188$ vectors $\avec$ in $\Ds{3}$, $\Ni{3}$, $\Nii{3}$ and $\Ntwo{3}$ that give curves of genus $4$ and $a$-number $3$ over $\F_3$. We give a summary of the number of vectors classified by case and some restrictions that occur in the case D. We use {\magma} to classify these curves in $\F_3$-isomorphism classes, which we detail in Table~\ref{table:isom_classes_p3} and Lemma~\ref{lem:standard_form_a3_p3_curves}.

	\begin{table}[h] 
		\caption{Isomorphism classes of genus $4$ and $a=3$ curves over $\F_3$.}
		\label{table:isom_classes_p3}
		\begin{center}
			\begin{tabular}{cccc}
				
				Case &$\#$ Isomorphism classes & $\#$ $p$-rank $1$& $\#$ $p$-rank $0$\\
				\hline 
				\hline 
				D&6 & 6&0 \\
				
				N1(i)	& 7& 7&0 \\
				
				N1(ii)	& 3& 3&0 \\
				
				N2	& 11& 11& 0\\
				\hline 
				Total	& 27& 27&0 \\ \hline
			\end{tabular}
		\end{center}
	\end{table}
	
	\begin{cor}\label{cor:exhaustive_search_p3}
		There are, up to $\F_3$-isomorphism, exactly $27$ curves of genus $4$ with $a=3$ in standard form.
	\end{cor}
	\begin{proof}
		See \tref{table:isom_classes_p3}.
	\end{proof}

	\begin{lem}\label{lem:standard_form_a3_p3_curves}
		Let $X$ be a curve in standard form with $a$-number $3$ defined over $\F_3$. 
		Then $X$ is isomorphic to one of the following:
		
		\begin{itemize}
			\item $V(F_d,G)$ with $G=x^3+y^3+xyz+c_1yz^2+xw^2+c_2w^3,$
			where $c_1 \in \F_3^\times$ and $c_2 \in \F_3$.
			
			\item $V(F_1,G)$ with $G=x^2y+c_1y^3+x^2z+c_2y^2w+c_3w^3+c_4z^3 +z^2w,$
			where $c_1 \in \{0,1\}$, $c_2,c_3 \in \F_3^\times$ and $c_4 \in \F_3$.
			
			\item $V(F_2,G)$ with $G=x^2y+c_1y^3+c_2x^2z+c_3z^3+c_4yzw+c_5w^3+c_6(z^2w+y^2w),$
			where $c_1,c_3 \in \F_3$, $c_2,c_4,c_5 \in \F_3^\times$ and $c_6 \in \{0,1\}$.
		\end{itemize}
	\end{lem}

	\subsubsection{Case $p=5$}

	From the exhaustive search we conclude that there are $134$ $\F_5$-isomorphism classes of standard form curves over $\F_5$ with $a$-number $3$. Table~\ref{table:isom_classes_p5} contains the summary of the isomorphism classes and the number of curves with $p$-rank $1$ and $p$-rank $0$.

	\begin{table}[h]  
		\caption{Isomorphism classes of genus $4$ and $a=3$ curves over $\F_5$.}
		\label{table:isom_classes_p5}
		\begin{center}
			\begin{tabular}{cccc}
				
				Case &$\#$ Isomorphism classes & $\#$ $p$-rank $1$& $\#$ $p$-rank $0$\\
				\hline 
				\hline 
				D&59 &46 &13 \\
				
				N1(i)	& 60& 48&12 \\
				
				N1(ii)	& 6& 4&2 \\
				
				N2	& 9& 0& 9\\
				\hline 
				Total	& 134& 98& 36\\ \hline
			\end{tabular}
		\end{center}
	\end{table}
	
	\begin{cor}\label{cor:exhaustive_search_p5}
		
		There are, up to $\F_5$-isomorphism, exactly $134$ curves of genus $4$ and $a=3$ in standard form.
		
	\end{cor}
	\begin{proof}
		See \tref{table:isom_classes_p5}.
	\end{proof}

	\subsubsection{Case $p=7$}

	The  exhaustive search for all curves in standard form with $a=3$ proved to be too time consuming for $p=7$. So it was only possible to find the curves in the case D and a subset of N1 curves.  We conclude from this search that there are at least $94$ $\F_7$-isomorphism classes of curves of $a$-number $3$, where $65$ correspond to the case D and $29$ to the N1. Now, we know by Lemma 4.5.1 in \cite{Kudo&Harashita} that any smooth, irreducible, genus $4$ curve $X=V(F_d,G)$ can be written in standard form. This is because condition (A3) is satisfied over $\F_7$ and $G$ can be reduced by a change of variables to the form of Definition~\ref{def:standardform}. This implies that the list of curves that we found in the case D actually includes all of the curves where the quadratic polynomial is degenerate. 
	%, here is how they break down:
	
	\begin{cor}\label{cor:degenerate_curves_over_F7}
		There are, up to $\F_7$-isomorphism, exactly $65$ genus $4$ and $a$-number $3$, smooth, irreducible, non-hyperelliptic curves over $\F_7$, given as $V(F,G)$, where $G$ is a cubic homogeneous polynomial and $F$ is a degenerate quadratic form. In addition, there are, up to $\F_7$-isomorphism, at least $29$ genus $4$ curve in standard form over $\F_7$ with $a$-number $3$ where $F$ is non-degenerate.
	\end{cor}
	\begin{proof}
		See Table~\ref{table:isom_classes_p7}.
	\end{proof}

	In the case D there are $1440$ curves with $a$-number $3$, which are divided into $65$ classes, $56$ of them have $p$-rank $1$ and $9$ have $p$-rank $0$. The subset of curves of the case N1i that we computed corresponds to those where the the cubic is of the form where $b_1=b_2=0$ (see Definition~\ref{def:standardform}). There are $432$ of these $G$ such that $V(F_1,G)$ is a smooth genus $4$ smooth curve with $a$-number $3$, only $16$ have $p$-rank $0$. These curves are distributed in $29$ ${\F}_7$-isomorphism class with $2$ of them having $p$-rank $0$ and $27$ of $p$-rank $1$.

	\begin{table}[h] 
		\caption{Isomorphism classes of genus $4$ and $a=3$ curves over $\F_7$.}
		\label{table:isom_classes_p7}
		\begin{center}
			\begin{tabular}{cccc}
				
				Case &$\#$ Isomorphism classes & $\#$ $p$-rank $1$& $\#$ $p$-rank $0$\\
				\hline 
				\hline 
				D&65 &56 &9 \\
				
				N1i'	& 29& 27&2 \\

				\hline 
				Total	&94 &83 &11 \\ \hline
			\end{tabular}
		\end{center}
	\end{table}
	
	\section{Cartier points }\label{sec:CartierPoints}
	
	\subsection{Definition and properties}\label{sect:CartierPoints_definitions and properties}

	Theorem~\ref{thm:Ekedahl} states that the genus of a superspecial curve in characteristic $p$ is bounded by $p(p-1)/2$. Ekedahl \cite{Ekedahl:Supersingular} bases the proof of this result on the fact that  a curve is superspecial if and only if its Jacobian is isomorphic to the product of supersingular elliptic curves (Oort, \cite{Oort75}). Baker, on the other hand, presents in \cite{Baker} an alternative proof which makes use of an equivalent definition: a curve is superspecial if and only if the Cartier operator annihilates $H^0(X,\Omega_X^1)$. The other component of his proof is the existence of linear systems of dimension $1$, associated to a certain type of points on $X$, defined as Cartier points. 
	
	\begin{definition}\label{def:cartier_points}
		A closed point $P$ of $X$ is said to be a Cartier point if the hyperplane of regular differentials vanishing at $P$ is stable under the Cartier operator.
	\end{definition}
	
	The following result yields Theorem~\ref{thm:Ekedahl} as a corollary.
	
	\begin{thm}[Baker \cite{Baker}, Theorem 2.8]\label{thm:baker}
		Let $X$ be a curve of genus $g$ over an algebraically closed field of characteristic $p$.
		\begin{enumerate}
			\item If $X$ has at least $p$ distinct Cartier points, no two of which differ by a $p$-torsion point on $J_X$, then $g \leq p(p-1)/2$.
			\item If $X$ is hyperelliptic of genus $g$, $p$ is odd, and some hyperelliptic branch point of $X$ is a Cartier point, then $g \leq (p-1)/2$.
		\end{enumerate}

	\end{thm}
	
	\begin{notation}\label{remark:on_notation}
		Let $k$ be a field of characteristic $p$. Let $X$ be a non-hyperelliptic curve of genus $g$ over $k$ embedded in $\PP(H^0(X,\Omega_X^1))=\PP^{g-1}$ by a basis \linebreak$\mathcal{B}'=\{\omega_1,\ldots, \omega_g\}$ of $H^0(X,\Omega_X^1)$. Suppose $x_1,\ldots, x_g$ are the coordinates of $\PP^{g-1}$ given by this basis and that $\mathcal{B}$ is the basis of $H^1(X,\mathcal{O}_X)$ dual to $\mathcal{B}'$. Given a point $P=[a_1:\cdots :a_g] $ of $X$ we denote by $\bv_P$ the vector $(a_1,\ldots, a_g)^T$ in $H^1(X,\mathcal{O}_X)$ expressed in terms of $\calB$.
		%	Conversely, given a vector $\bw=(b_1,\ldots, b_g)^T$ in $H^1(X,\mathcal{O}_X)$ expressed in terms of $\calB$, let $Q_\bw$ denote the point $[b_1:\ldots :b_g]$. 
	\end{notation}
	From here on, we consider the Hasse--Witt matrix $H$ of $X$ to be in terms of the basis $\calB$, unless otherwise stated.

	\begin{prop}\label{prop:find-c-pts-with-HW}
		Let $X$ be a non-hyperelliptic curve of genus $g$ over $k$ embedded in $\PP^{g-1}$ as in Notation~\ref{remark:on_notation}. A point $P=[a_1:\cdots :a_g] $ of $X(\overline{k})$ is a Cartier point if and only if there exists $c\in \overline{k}$ such that
		\begin{equation}\label{eq:Hv^p=cv_is CP}
		H\bv_P^{(p)}=c\bv_P,
		\end{equation}
		where $\bv_P^{(i)}$ indicates that each entry of the vector is raised to the $i$-th power.
	\end{prop}
	
	\begin{proof}
		Since $X$  is embedded in $\PP^{g-1}$ by $\{\omega_1,\ldots, \omega_g\}$ then $\omega_i(P)=a_i$, for  \linebreak $i=1,\ldots, g$. A regular $1$-form $\omega=b_1\omega_1+\cdots + b_g\omega_g$ vanishes at $P$ if and only if 
		\begin{equation}
		b_1a_1+\cdots + b_ga_g=0.
		\end{equation}
		Then the hyperplane of $1$-forms vanishing at $P$ is 
		\begin{equation}
		L_P:= \left\lbrace (b_1,\ldots ,b_g)^T:  b_1a_1+\cdots + b_ga_g=0\right\rbrace.
		\end{equation}
		Let $L_P^0$ be the annihilator of $L_P$. We know that this is a $1$-dimensional subspace of ${H^0(X,\Omega_X^1)}^{*}\cong H^1(X,\calO_X)$, so it is generated by the vector $\bv=(a_1,\ldots, a_g)^T$.
		
		Now, $L_P$ is stable under the action of $\calC$ if $\mathcal{C}(\omega)\in L_P$, for every $\omega \in L_P$. By duality, this is equivalent to $\calF_X(\bv_P)=c\bv_P$, that is, $H\bv_P^{(p)}=c\bv_P.$
		
	\end{proof}
	
	\begin{definition}\label{def:T1andT2}
		
		We say that a Cartier point $P\in X$ as above is of\textit{ Type $1$} if $c=0$ and of \textit{Type $2$} otherwise. 
		
	\end{definition}

	Notice that if $k$ is not algebraically closed, then the Cartier points of $X/k$ might not be defined over $k$, but over some extension. In general we consider Cartier points as points in $X(\overline{k})$.  The next lemma gives us a way to find Type $1$ points.

	\begin{lem}\label{lem:Type_1Points_H1/pv=0}
		A point $P$ is a Type $1$ Cartier point of $X$ if and only if \linebreak$H^{(1/p)}\bv_P=0$.
	\end{lem}
	\begin{proof}
		By definition $P$ is a Type $1$ point of $X$ if and only if $H\bv^{(p)}=0$. By applying the inverse of the $p$-th Frobenius morphism we see that this is equivalent to $H^{(1/p)}\bv_P=0$.
	\end{proof}

	\begin{cor}\label{cor:find_type1_points_Hv=0_Fp}
		Supose $X$ is defined over $\F_p$. Then $P\in X$ is a Type $1$ point if and only if $H\bv_P=0$.
	\end{cor}
	
	Suppose that $X$ is defined over $\F_{q}$, with $p^r=q$ for some positive integer $r$. For a point $Q\in X$ we denote by $\sigma(Q)$ the action of the $r$-th power of the Frobenius morphism on $Q$.% It turns out that the set of Cartier points is stable under the action of $\sigma$. 
	\begin{lem}\label{lem:frob_cartierpoint}
		Let $Q$ be a Cartier point of $X/\F_q$ and let $P=\sigma(Q)$. Then $Q$ is a Cartier point of $X$ if and only if $P$ is a Cartier point of $X$. Furthermore, if $H\bv_Q^{(p)}=c\bv_Q$, then $H\bv_P^{(p)}=c^{q}\bv_P$.
	\end{lem}
	\begin{proof}
		
		First we note that $P$ is also a point on $X$, because $X=V(F,G)$ and $F$ and $G$ are defined over $\F_q$. Let $H$ be the Hasse--Witt matrix of $X$. After scaling, we can assume that $\bv_Q^{(q)}=\bv_P$. Then

		$$\begin{array}{ccl}
		H\bv_Q^{(p)}=c\bv_Q& \Leftrightarrow&(H\bv_Q^{(p)})^{(q)}=(c\bv_Q )^{(q)}\\
		
		&\Leftrightarrow& H(\bv_Q^{(q)})^{(p)}=c^{q}(\bv_Q)^{(q)}\\
		
		&\Leftrightarrow & H\bv_P^{(p)}=c^{q}\bv_P,
		\end{array}$$

		where the second equivalence is true because $X$ is defined over $\F_q$ then so is $H$. 
	\end{proof}

	We apply Lemma~\ref{lem:frob_cartierpoint} to reduce the search of Cartier points to a computation of eigenvectors. 
	
	\begin{lem}\label{lem:T_points_are_eigenvectors}
		Suppose $X$ is defined over $\F_p$ with  Hasse--Witt matrix $H$. Let $Q$ be a Cartier point of $X$ defined over $\F_{p^e}$ for some positive integer $e$. There exists $\lambda \in \F_{p}$ such that $H^e\bv_Q=\lambda\bv_Q$. 
	\end{lem}
	
	\begin{proof}
		Since $Q$ is a Cartier point in of $X$ then $H\bv_Q^{(p)}=c\bv_Q$ for some $c\in \Fbar_p$. Now, since $Q\in X(\F_{p^e})$, then $c\in\F_{p^e}$.	Let $P_i:=\sigma^i(Q)$, then $P_1, P_2, \ldots, P_e=Q$ are distinct Cartier points. Also, after scaling we can assume that $\bv_{P_i}^{(p)}=\bv_{P_{i+1}} $. After applying $i$ times the result from Lemma~\ref{lem:frob_cartierpoint} with $q=p$, we get that $H\bv_{P_i}^{(p)}=c^{p^i}\bv_{P_i} $. To ease notation, we write $\bv_i$ instead of $\bv_{P_i}$. Then
		\begin{eqnarray*}
			H^e\bv_e=&H^e\bv_{e-1}^{(p)}=H^{e-1}\left(H\bv_{e-1}^{(p)}\right)= H^{e-1}\left(c^{p^{e-1}}\bv_{e-1}\right)= c^{p^{e-1}}H^{e-1}\bv_{e-2}^{(p)}.
		\end{eqnarray*} 
		By an inductive process, we get that 
		$H^e\bv_Q= c^{p^{e-1}+p^{e-2}\cdots +p+1}\bv_Q=c^{\frac{p^e-1}{p-1}}\bv_Q$. Let $\lambda:=c^{\frac{p^e-1}{p-1}}$. Since $c\in \F_{p^e}$ then $\lambda$ is a $(p-1)$-root of unity, if $c\neq 0$ and $\lambda=0$ if $c=0$. Hence $\lambda \in \F_p$.
		
	\end{proof}

	Baker provides in \cite{Baker} an upper bound for the number of Cartier points on a smooth irreducible curve that is not ordinary nor superspecial.

	\begin{prop} [Baker, \cite{Baker} Prop. 3.3]\label{prop:bound_cartier_points}
		Let $X$ be a smooth, irreducible curve of genus $g$ with $p$-rank $f$, which is not ordinary nor superspecial. 
		\begin{enumerate}
			\item The number of Type $2$ points on $X$ is bounded by
			\begin{align*}\label{eq:bound-on-CP}
			b:=b_{g,p,f,\delta_X}=	\min \left(2g -2, \delta_X\frac{p^f-1}{p-1} \right),
			\end{align*}
			where $\delta_X$ is $2$ if $X$ is hyperelliptic and $1$ otherwise.
			\item The number of Type $1$ points on $X$ is bounded by $2g-2$. Furthermore, if the $a$-number of $X$ is $g-1$ then there is at least one Type $1$ Cartier point on $X$.
		\end{enumerate}
	\end{prop}

	Let us explore the geometric meaning of this bound, following the proof of Proposition~\ref{prop:bound_cartier_points} in \cite{Baker}. We will only be concerned with non-hyperelliptic curves, so here $\delta_X=1$.
	
	We know by Proposition~\ref{prop:find-c-pts-with-HW} that $P\in X(\kbar)$ is a Cartier point if and only if there exists $c \in \kbar$ such that $H\bv_P^{(p)}=c\bv_P$. If $c=0$ then this equation is equivalent to $H\bv_P^{(1/p)}=0$. Hence the Type $1$ points are those in the intersection of $X(\kbar)$ and the subspace spanned by the kernel of $H^{(1/p)}$ in $\PP^{g-1}$.  
	This subspace is linear and has codimension at least $1$, so it is contained in a hyperplane. From where we conclude that the number of points in the intersection is at most the degree of the curve $2g-2$. 
	
	On the other hand, if $c\neq 0$ then we can rewrite Equation \eref{eq:Hv^p=cv_is CP} as $H\bw^{(p)}=\bw$ by setting $\bw=\lambda\bv_P$
	for $\lambda$ equal to some $(p-1)$-th root of $c^{-1}$. Now the element of $H^1(X,\calO_X)$ given by $\bw$ is fixed by the Frobenius operator. By definition the $p$-rank $f$ is the dimension of the subspace of $\HOOX$ where $\calF$ is bijective. So there are $p^f-1$ non trivial solutions, that yield at most $\frac{p^f-1}{p-1}$ Type $2$ points. After maybe doing a base extension on $X$ one can choose a basis of $H^0(X,\Omega_X^1)$ given by $\{\xi_1, \ldots, \xi_g\}$ such that $\calC(\xi_i)=\xi_i$ for $1\leq i\leq f$. Assume that the coordinates of a Type $2$ point are given by  this basis. Then,some coordinate, say $x_g$ of every such point must be zero, since $f<g$. Then the point lies on the hyperplane $x_g=0$. Again, there can only be $2g-2$ such points, so the number of Type $2$ points is $\min \{ 2g-2,\frac{p^f-1}{p-1} \}$.

	As a direct consequence of the existence of the upper bound on the number of Cartier points, we also get a bound on the degree of the field of definition of the point. In particular, we have Corollary~\ref{cor:degree_of_definition_C_points} for the case $k=\F_q$.

	\begin{cor}\label{cor:degree_of_definition_C_points}
		Let $X$ be a genus $g$ non-hyperelliptic curve that is not ordinary nor superspecial, defined over $\F_q$. 
		\begin{itemize}
			\item If $P$ is a Type $1$ Cartier point of $X$, then $P\in X(\F_{q^e})$ with $1\le e \le 2g-2$.
			\item If $P$ is a Type $2$ Cartier point of $X$, then $P\in X(\F_{q^e})$ with $1\le e \le b$, where $b$ is as in Proposition~\ref{prop:bound_cartier_points}.
		\end{itemize}

	\end{cor}
	\begin{proof}
		Let $P$ be a Cartier point of $X$. Let $e$ be the minimum positive integer such that $P\in X(\F_{q^e})$. By Lemma~\ref{lem:frob_cartierpoint}, the $e$ distinct points \linebreak $\{P, \sigma(P), \sigma^2(P), \ldots, , \sigma^{e-1}(P)\}$ are all Cartier points of the same type. By Proposition~\ref{prop:bound_cartier_points}, there are at most $2g-2$ Type $1$ points and $b$ Type $2$ points. Hence, if $P$ is a Type $1$ point (resp. Type $2$), then $e\leq 2g-2$ (resp. $e\leq b$).

	\end{proof}
	
	\subsubsection{Type $1$ points in the case $a=g-1$} \label{sec:T1_a=g-1}
	The behavior of the Cartier points when the $a$-number is $g-1$ has an additional feature which is the multiplicity. In this case, the subspace $S$ generated by the kernel of $H^{(1/p)}$ is a hyperplane, then assuming $X\nsubseteq S$, the intersection $X \cap S$ is proper. Then we can define the intersection multiplicity of $P$ in $ X \cap S$, that is, of the Type $1$ Cartier points. 
	
	If $X\cap S=\{P_1,\ldots, P_n\}$ and $m_i$ denotes the multiplicity of the point $P_i$, then $\sum_{i=1}^{n}m_i=2g-2 $. Hence the possible multiplicity distributions of the Type 1 points correspond to the partitions of $2g-2$. Both the multiplicity and the degree of $P$ are preserved under $\sigma$, so if $d_i$ is the degree of the point $P_i$, then $P_i, \sigma(P_i),\sigma^2(P_i),\ldots, \sigma^{d_i-1}(P_i)$ are $d_i$ distinct points of the same degree and multiplicity.

	\subsection{Cartier points on genus $4$ curves}\label{sec:CartierPoints_g4_curves}

	Suppose $q=p^r$ for some positive integer $r$ and  let $X=V(F,G)$ be a non-ordinary and non-superspecial smooth, irreducible genus $4$ non-hyperelliptic curve over $\F_q$.  We want to determine the sharpness of the bound given by Baker in Proposition~\ref{prop:bound_cartier_points} for the number of Cartier points on $X$. We will first make some remarks about the possible bounds depending on the $a$-number  and the $p$-rank. 
	
	\begin{cor}[Type $2$ Cartier points and the $p$-rank]\label{cor:type2CP}
		Let $X$ be a non-hyperelliptic curve of genus $4$ defined over $\F_q$ with $p$-rank $f$.  
		\begin{enumerate}
			\item [(i)] There are at most six Type $2$ points on $X$ and they are defined over $\F_{q^e}$ for some $1\le e \le 6$. Moreover:
			\item[(ii)] if $f=0$, then there are no Type $2$ points;
			\item[(iii)] if $f=1$, then there is at most one Type $2$ point and it must be defined over $\F_q$;
			\item[(iv)] if $f=2$, then there are at most six Type $2$ points, at most three if $p=2$ and at most four if $p=3$.
		\end{enumerate}
	\end{cor}

	\begin{proof}
		Parts \textit{(i)}, \textit{(ii)} and \textit{(iv)} are direct consequences of the discussion following Proposition~\ref{prop:bound_cartier_points}. Indeed, the minimum of $2g-2=6$ and $\frac{p^f-1}{p-1}$ is 6, unless $f\leq 1$  or $p=2,3$ and $f= 2$.

		For \textit{(iii)} let $Q$ be a Type $2$ Cartier point of $X$ and suppose that $\sigma(Q)\neq Q$. By Lemma~\ref{lem:frob_cartierpoint} we have that $\sigma(Q)$ is another Cartier point of Type 2, which is a contradiction.  
	\end{proof}

	\begin{cor}[Type $1$ points and the $a$-number]\label{cor:type1_genus4}
		If $X$ is non-hyperelliptic of genus $4$ with $a$-number $1,2$ or $3$ defined over $\F_q$, then 
		\begin{enumerate}
			\item[(i)] $X$ has at most six Cartier points of Type $1$ and they are defined over $\F_{q^e}$ for some $e\leq 6$. Moreover,
			\item[(ii)]  if $a=1$, then there is at most one Type $1$ Cartier point
			\item[(iii)] if $a=3$, there are exactly six Type $1$ Cartier points on $X$, counting with multiplicity.
		\end{enumerate}
		
	\end{cor}
	\begin{proof}
		Part \textit{(i)} follows from Proposition~\ref{prop:bound_cartier_points} and 
		Corollary~\ref{cor:degree_of_definition_C_points}, \linebreak using $2g-2=6$. 
		Now, if $a=1$, the kernel of $H^{(1/p)}$ has dimension $1$, because $a$ equals $\dim(\ker(H))$. Then $\ker(H)$ spans a point in $\PP^3$, and this is the only possible Type $1$ point. 
		When $a=3$, on the other hand, the subspace spanned by the same kernel is a hyperplane, so its intersection with $X$ consists on exactly as many points as the degree of the curve. 
	\end{proof}

	\subsubsection{Computing Cartier points}

	We explain here an algorithm to find the Cartier points on smooth, irreducible genus $4$ non-hyperelliptic curves over $\F_p$, given their quadratic and cubic defining homogeneous polynomials. Before we detail the procedure, we will revisit some facts about Cartier points. Let $X=V(F,G)$ be a curve as before.  
	\begin{enumerate}
		\item Let $H$ be the Hasse--Witt matrix of $X$ as in Proposition~\ref{prop:computation_hw}. Recall that $H$ represents the action of Frobenius on $H^1(X,\calO_X)$ with respect to the basis that corresponds to the coordinates $x,y,z,w$. This implies that we can use $H$ and Proposition~\ref{prop:find-c-pts-with-HW} to find the Cartier points of $X$, by solving $H\bv^{(p)}=c\bv$.
		
		\item If $P \in X(\F_{p^e})$ is a Cartier point, then there exists some  $\lambda \in \F_{p}$ such that $H^e\bv_P=\lambda\bv_P$ (Lemma~\ref{lem:T_points_are_eigenvectors}).
		\item If $P$ is a Type $1$ point, then $H\bv_P=0$ (Corollary~\ref{cor:find_type1_points_Hv=0_Fp}).
		\item The eigenvalues of $H^e$ are exactly the $e$-powers of the eigenvalues of $H$. If $\lambda $ is as in (2), then there exists an eigenvalue $\mu$ of $H$ such that $\lambda=\mu^e\in \F_{p}$.%In other words, in order for a Cartier point to be defined over $\F_{p^e}$, then there must be a $\mu$ such that $\mu^e \in \F_{p}$. 
		\item Let $h(x)$ be the characteristic polynomial of $H$. The splitting field of $h(x)$ is either $\F_p$, $\F_{p^2}$ or $\F_{p^3}$. Indeed, $h(x)$ has degree $4$, but since $H$ has rank at most $3$, then $x$ is a factor of $h(x)$. 
		
	\end{enumerate}
	
	We use these facts to compute the Type $1$ and Type $2$ Cartier points of $X$.  Algorithm~\ref{alg:Type2_points} is restricted to the case when $q=p$, to simplify the computations.
	
	\begin{algo}\label{alg:Type1_points}[Type $1$ Cartier points]
		\noindent
		Input: $F$ and $G$ in $\F_q[x,y,z,w]$.\\
		Output: List of Type $1$ Cartier points of $X=V(F,G)$.
		\begin{enumerate}
			\item Compute the Hasse--Witt matrix $H$ of $X$, as in Proposition~\ref{prop:computation_hw}.
			\item Let $M=H^{(1/p)}$. Construct the linear forms $L_i=(M)_{i,1}x+(M)_{i,2}y+(M)_{i,3}z+(M)_{i,4}w$, for $1\leq i\leq 4$.
			\item Let $I$ be the ideal generated by $\{L_1,L_2,L_3,L_4,F,G\}$ and let $T=V(I)$.
			\item For each  $1\le e \le 6$,  find the points in $T_e=T(\F_{q^e})$.
			\item The set of Type $1$ points is $\bigcup_{e}(T_e)$.
		\end{enumerate}
	\end{algo}

	\begin{algo}\label{alg:Type2_points}[Type $2$ Cartier points]
		
		\noindent
		Input: $F$ and $G$ in $\F_p[x,y,z,w]$.\\
		Output: List of Type $2$ Cartier points of $X=V(F,G)$.
		
		\begin{enumerate}
			\item Compute the Hasse--Witt matrix $H$ of $X$, as in Proposition~\ref{prop:computation_hw}.
			\item Compute $h(x)$, the characteristic polynomial of $H$ and find the roots of $h(x)$ in its splitting field. For each $\mu$ non-zero root of $h(x)$ and  each $1\le e \le 6$ such that $ \mu^e \in \F_{p}$:
			\begin{enumerate}
				\item Let $M=H^e-\mu^eI$, where $I$ is the identity matrix. 
				\item  For $1\leq i\leq 4$, construct the linear forms $L_i=(M)_{i,1}x+(M)_{i,2}y+(M)_{i,3}z+(M)_{i,4}w$.
				\item Let $I$ be the ideal generated by $L_1,L_2,L_3,L_4,F,G$ and let $T=V(I)$.
				\item For each point in $T(\F_{p^e})$, compute $H\bv_P^{(p)}$. If this gives a scalar multiple of $\bv$, then $P$ is a Type $2$ Cartier point.
				
			\end{enumerate}
		\end{enumerate}

	\end{algo}

	\subsection{Cartier points on standard curves over $\F_p$}\label{sec:Cartier_points_on_stn_FP}
	In this section we present the results from our search of curves in standard form, related to their $a$-number, $p$-rank and Cartier points, both of Type $1$ (T1) and Type $2$ (T2). 
	
	\textbf{The upper bound of Type $1$ points: }Our data reflects that, as expected, it is hard to find curves that attain the bounds on the number of Type $1$ Cartier points. For instance when the $a$-number is 1, a curve can have at most one Type $1$ point, but most curves have zero (see Table~\ref{tab:summary_T1}).  When $a=2$ the bound is six, but as stated in Corollary~\ref{cor:no_a2_curves_reach_UB_of_T1}, all of the curves in our sample have three or fewer Type $1$ Cartier points. 
	
	The bound on Type $1$ points is also six when $a=3$, but here we see a different behavior. This is mainly because in this case, the Type $1$ points come from intersecting the curve with a hyperplane, which guarantees exactly six points, counting with multiplicity. Even though our statistical data does not give a large sample of $a$-number $3$ curves over $\F_{7}$ and $\F_{11}$, we can see in Tables~\ref{tab:method1_curves} and~\ref{tab:summary_T1} that most curves attain the bound, except when $p=3$.

	\textbf{The upper bound on Type $2$ points:}
	Similarly, it seems unlikely for a curve with $p$-rank $f>0$ to reach the upper bound of Type $2$ points. This is true even when $f=1$ and thus the bound is one. 
	We can see in Table~\ref{tab:summary_T2} that the majority of curves with $p$-rank $1$ have no such points. The same happens when $f=2$. In this case the bound is six, but we did not find any curve with more than two Type $2$ points, except when $p=5$, where there exist two curves with six Type $2$ points.
	
	It is important to recall that the bound for Type $2$ points in characteristic $3$ is $4$, not $6$, and there are in fact $16$ $p$-rank $3$ curves that reach this bound. 	
	
	Here we present the overall results from obtained from our data. The details corresponding to each of $p \in \{3,5,7,11\}$ are given in the following subsections.
	
	%\catalina{A comment here would be nice}
	\begin{cor}\label{cor:no_a2_curves_reach_UB_of_T1}
		In our sample of smooth curves, no curve with $a$-number $2$ reaches the bound of six Type $1$ Cartier points. Moreover, the maximum number of Type $1$ Cartier points attained for curves with $a$-number $2$ is three for $p \in\{5,7,11\}$ and two for $p=3$.
	\end{cor}
	\begin{proof}
		See Table~\ref{tab:summary_T1}.
	\end{proof}

	\begin{cor}\label{cor:summary_T1_bounds}
		The bound on the number of Type 1 points is sharp for non-hyperelliptic smooth genus $4$ curves over $\F_p$ when
		\begin{itemize}
			\item $p\in \{3,5,7,11\}$ and $a=1$.
			\item $p\in \{5,7,11\}$ and $a=3$.
		\end{itemize} 
	\end{cor}
	
	\begin{proof}
		See Table~\ref{tab:summary_T1}.
	\end{proof}

	\begin{cor}\label{cor:no_f2_curves_reach_UB_of_T2}
		In our sample of smooth curves, no curve with $p$-rank $2$ or $3$ reaches the bound of six Type $2$ Cartier points when $p\in \{5,7,11\}$.

	\end{cor}
	\begin{proof}
		See Table~\ref{tab:summary_T2}.
	\end{proof}

	\begin{table}[h]  
		\caption{Summary of Type $1$ points on samples of standard form curves.}\label{tab:summary_T1}
		\begin{center}
			\begin{tabular}{cccccc}
				$a$-number	&	\#T1 	&	3	&	5	&	7	&	11	\\	\hline\hline
				1	&	0	&	23971	&	88896	&	86563	&	25205	\\	
				1	&	1	&	5327	&	3896	&	1918	&	202	\\	
				\multicolumn{2}{c}{Total}		&	29298	&	92792	&	88481	&	25407	\\	
				\multicolumn{2}{c}{\% that attains UB}		&	18.18	\%&	4.20	\%&	2.17	\%&	0.80	\%\\	\hline
				2	&	0	&	1690	&	3158	&	1595	&	191	\\	
				2	&	1	&	3268	&	832	&	262	&	25	\\	
				2	&	2	&	1673	&	149	&	37	&	1	\\	
				2	&	3	&	0	&	65	&	5	&	1	\\	
				
				\multicolumn{2}{c}{Total}		&	6631	&	4204	&	1899	&	218	\\	
				\multicolumn{2}{c}{\% that attains UB}		&	0	\%&	0	\%&	0	\%&	0	\%\\	\hline
				3	&	0	&	0	&	0	&	0	&	0	\\	
				3	&	1	&	660	&	4	&	0	&	0	\\	
				3	&	2	&	40	&	4	&	0	&	0	\\	
				3	&	3	&	0	&	14	&	0	&	0	\\	
				3	&	4	&	0	&	0	&	1	&	0	\\	
				3	&	5	&	0	&	0	&	0	&	0	\\	
				3	&	6	&	0	&	24	&	4	&	1	\\	
				\multicolumn{2}{c}{Total}		&	700	&	46	&	5	&	1	\\	
				\multicolumn{2}{c}{\% that attains UB}		&	0	\%&	52.17	\%&	80.00	\%&	100	\%\\	\hline
				&	Sample size	&	186266	&	719102	&	863038	&	361098	\\	\hline
			\end{tabular}
		\end{center}
	\end{table}

	\begin{table}[h]  
		\caption{Summary of Type $2$ points on samples of standard form curves. }\label{tab:summary_T2}
		\begin{center}
			\begin{tabular}{cccccc}
				$f$	&	\#T2	&	3	&	5	&	7	&	11	\\	\hline\hline
				0	&	0	&	1157	&	193	&	44	&	3	\\	\hline
				1	&	0	&	4276	&	4272	&	1853	&	239	\\	
				1	&	1	&	198	&	146	&	35	&	0	\\	
				\multicolumn{2}{c}{Total}		&	4474	&	4418	&	1888	&	239	\\	
				
				\hline
				2	&	0	&	7137	&	17072	&	12071	&	2320	\\	
				2	&	1	&	353	&	737	&	235	&	20	\\	
				2	&	2	&	23	&	35	&	5	&	0	\\	
				2	&	3	&	0	&	2	&	0	&	0	\\	
				
				\multicolumn{2}{c}{Total}		&	7513	&	17846	&	12311	&	2340	\\	
				
				\hline
				3	&	0	&	21951	&	71423	&	74504	&	22837	\\	
				3	&	1	&	1394	&	3032	&	1597	&	204	\\	
				3	&	2	&	117	&	113	&	41	&	3	\\	
				3	&	3	&	14	&	15	&	0	&	0	\\	
				3	&	4	&	9	&	2	&	0	&	0	\\	
				
				\multicolumn{2}{c}{Total}		&	23485	&	74585	&	76142	&	23044	\\	
				
				\multicolumn{2}{c}{Sample size}			&	186266	&	719102	&	863038	&	361098	\\	\hline
			\end{tabular}
		\end{center}
	\end{table}

	We know by Corollary~\ref{cor:type1_genus4}  that the upper bound on the number of Type $1$ Cartier points is given by $2g-2=6$ if $a=2,3$ and $1$ if $a=1$. We want to determine for which $a,f$ and $p$  these bounds are attained. For example, this does not happen for any of the curves with $a=2$ that we found.  In Table~\ref{tab:summary_T1}, we break down the number of curves over $\F_p$ by the number of Type $1$ points they have.

	On the other hand, the upper bound on the number of Type $2$ points depends on the $p$-rank: it is $\min \{2g-2, \frac{p^f-1}{p-1}\}$. As stated in Corollary~\ref{cor:no_a2_curves_reach_UB_of_T1}, and observed in Table~\ref{tab:summary_T2}, it is unlikely that a curve of $p$-rank $2$ or $3$ reaches the bound of six Type $2$ points.

	\subsubsection{Case $p=3$}\label{sec:CP_on_GC_p=3}
	
	From the sample of $186266$ curves over $\F_3$ a total of $35472$ have $a$-number $1,2$ or $3$.  There are only two instances in which we identified curves that realize the (non-zero) upper bound on Type $1$ and Type $2$ points. These are when $a=1$, for Type $1$ points, and when $f=3$ for Type $2$ points, and the bounds are $1$ and $4$, respectively. Even so, we can observe in Tables~\ref{tab:summary_T1} and ~\ref{tab:summary_T2} that most curves tend to have fewer Cartier points.

	\subsubsection{Case $p=5$}
	Recall that we sampled a total of $719102$ tuples over $\F_5$ and obtained $97042$ curves of $a$-number $1\leq a \leq 3$. %As expected, most of these curves have  $a$-number $1$. 
	Only $3896$ out of the $88896$ curves with $a=1$ have  a Type $1$ Cartier point.  This is expected because for each curve, there is only one point $P$ such that $H\bv_P=0$, so it is unlikely for this point to also be on the curve.

	\subsubsection{Case $p=7$}\label{subsub:CPsF7}
	We sampled a total of $863200$ random tuples, obtaining $619872$ curves in standard form and $90403$ of them with $a$-numbers $1,2$ or $3$. We note that the upper bounds for the number of Type $1$ points are realized when $a=1$ and $a=3$, but not for $a=2$, where the maximum number of points attained is three. With respect to the Type $2$ points, the (non-zero) upper bounds are only attained when the $p$-rank is $1$. In particular, we get the following result:
	
	\begin{cor}\label{cor:bounds_reached_p7}
		Baker's bound on the total number of Cartier points for genus $4$ curves with $a=3$ and $p$-rank $1$ is attained over $\F_7$.
	\end{cor}
	
	In our data there is only one curve with $a$-number $3$ where both the bounds of Type $1$ and Type $2$ points are attained, we show it in Example~\ref{ex:CPbound_reached_p7_2}. We also include Example~\ref{ex:p7_6T1_points_2}, in which we see a curve with only $4$ Type $1$ points. The Hasse--Witt matrix in each example is computed with the basis given in Proposition~\ref{prop:computation_hw}.

	\subsubsection{Case $p=11$}\label{subsub:CPsF11}
	
	The size of the random sample in this case is of $361098$ tuples. Note that, once again, the upper bound for the Type $1$ points is attained for some cases where $a=1,3$. On the other hand, none of the curves realize the bound of Type $2$ points (except, of course when $f=0$). One important observation is that, from all sampled tuples, only one of them resulted in a curve with $a$-number $3$. This curve also achieves the maximum of $6$ Type $1$ points.

	\section{Cartier points on curves over $\F_p$ with $a$-number $3$}
	
	In Table~\ref{tab:method1_curves} we display the information of curves in standard form over $\F_p$ for $p \in \{3,5,7\}$, obtained from the exhaustive search. For $p=3,5$ the search was done over all the possible tuples. For $p=7$ we found all the curves in standard form in the degenerate case and all of those in the case N1i when the coefficients of $y^2z$ and $z^2w$ are 0. We denote these lists by 7(D) and 7(N1i'), respectively.
	In the second column, we list the total number of isomorphism classes of curves with $a$-number $3$ found for each $p$. The third column indicates the total number of curves that attain the maximum of six Type $1$ Cartier points. The last two columns correspond to the number of curves that reach the maximum of Type $2$ points, over the total with the respective $p$-rank. Notice that all curves with $f=0$ trivially reach this bound, since the bound is $0$.

	\begin{table}[h]  
		\caption{Isomorphism classes in standard form with $a$-number $3$ over $\F_p$. }\label{tab:method1_curves}
		\begin{center}
			\begin{tabular}{c|c|c|cc}
				
				\multirow{3}{*}{$p$	}&Classes & { Attain max. of T1} &\multicolumn{2}{c}{Attain max. of T2}  \\ 
				&$a=3$	&$a=3$&$f=0$&$f=1$\\	
				&		&	$(6)$	&	$(0)$&$(1)$	\\
				\hline
				\hline	
				{3}			&27&0&0/0&0/27\\	
				
				{5}		&134&80&36/36&5/98\\
				
				{7(D)}	&65&48&9/9&0/56\\	
				
				{7(N1i')}	&29&23&2/2&1/27\\	
				
				\hline
			\end{tabular} 
		\end{center}
	\end{table}

	In the case $p=3$ all curves have either one or two Type $1$ points maximum; for $p=5$, the curves show one, two three or six points, and for $p=7$ we saw evidence of curves with any number of Type $1$ points ranging from one to six. In addition, over $\F_5$, every possible degree distribution for curves with six (distinct) Type $1$ points occurs. Also, all except one of them occur among our sample over $\F_7$.

	\subsection{Case $p=3$}

	We will now discuss the curves in standard form over $\F_3$ of $a$-number $3$. Since we have a complete list of all of the isomorphism classes of these curves, we know that all of them have $p$-rank $1$.

	In Lemma~\ref{lem:standard_form_a3_p3_curves} we specify representatives for the isomorphism classes of curves in standard form over $\F_3$ with $a$-number $3$. It turns out the all the curves from the same kind have the equal multiplicity and degree distribution of Type $1$ points, as a consequence we get the next lemma. We show explicitly the points in Tables ~\ref{table:repres_p3_D},~\ref{table:repres_p3_N1} and~\ref{table:repres_p3_N2}.

	\begin{lem}\label{lem:T1_points_in_a3_p3_curves}
		A curve $X=V(F,G) $ in standard form with $a$-number $3$ over $\F_3$ has no Type $2$ Cartier points and
		\begin{itemize}
			\item if $F=F_d=2yw+z^2$, then $X$ has exactly one Type $1$ point of multiplicity 6;
			\item if $F=F_1=2xw+2yz$, then $X$ has exactly two Type $1$ points of multiplicity 3, each defined over $\F_3$;
			\item if $F=F_2=2xw+y^2+z^2$, then $X$ has exactly two Type $1$ points of multiplicity 3, each defined over $\F_9.$
		\end{itemize}
	\end{lem}
	
	\begin{proof}
		See Tables~\ref{table:repres_p3_D},~\ref{table:repres_p3_N1} and~\ref{table:repres_p3_N2}.
	\end{proof}
	
	The proof of Lemma~\ref{lem:T1_points_in_a3_p3_curves} follows from our classification of curves in standard form into isomorphism classes and direct computation of Cartier points. Here we give an overview of the heuristics that go into this procedure, in the degenerate case. The other two cases can be worked out in a similar way.

	\begin{example}[Cartier points in the case D]\label{ex:T1_T2_on_D_p3_a3}
		Let $X=V(F,G)$ be a curve from Table~\ref{table:repres_p3_D}.  
		We know that  $F=2yw+z^2$ and that  $G$ is of the form
		$$ G=x^3+y^3+xyz+c_1yz^2+xw^2+c_2w^3,$$
		where $c_1 \in \F_3^\times$ and $c_2 \in \F_3$. We will show that $X$ has no Type $2$ Cartier points and that $[2-c_1-c_2:1:2:1]$ is the  unique Type $1$ Cartier point.

		By Proposition~\ref{prop:computation_hw}, the Hasse-Witt matrix of $X$ is
		\begin{equation}
		H= \begin{bmatrix}
		0& 0& 0 &0\\
		0&1 & 2 & 1\\
		0& 0 & 0 & 0\\
		0&0 & 0&0 \\	 
		\end{bmatrix}.
		\end{equation}
		We can see right away that $X$ has no Type $2$ Cartier points, since, up to scalar multiplication, the only vector such that $H\bv=c\bv$ for some $c\in \F_3$ is \linebreak $\bv=(0,1,0,0)^T$, but these are not the coordinates of a point on $X$. 
		
		Now, the hyperplane $S$ generated by the kernel of $H$ is the zero locus of $L= y+2z+w$. Then by substituting $y=z-w$ in $F=0$ we get $z^2+2zw+w^2=0$. If $w=0$, then $z=y=0$, but $[1:0:0:0] \not\in X$, so assume $w=1$. Then $z=2$ and $y=1$ and the Type $1$ points of $X$ are of the form $P=[\alpha:1:2:1]$. To find $\alpha $ we evaluate $G$ at $P$ and get $\alpha =2-c_1-c_2$.  So there is a unique Type $1$ point, which must have multiplicity $6$. 

	\end{example}

	In Tables~\ref{table:repres_p3_D},~\ref{table:repres_p3_N1} and~\ref{table:repres_p3_N2} we show the cubic polynomial for a representative of each isomorphism class in the cases D, N1 and N2, respectively. The curves are of the form $X=V(F,G)$. We also specify the Type $1$ Cartier points.

	\begin{table}[h]  
		\caption{Isomorphism classes of D curves with $a=3$ over $\F_3$.}
		\label{table:repres_p3_D}

		\begin{center}
			\begin{tabular}{cc}
				
				Representative $G$ & Type 1 point\\
				\hline 
				\hline
				$x^3 + y^3 + xyz + yz^2 + xw^2$	& $[ 1 : 1 : 2 : 1] $\\
				
				$x^3 + y^3 + xyz - yz^2 + xw^2$	&  $ [0 : 1 : 2 : 1]$ \\
				
				$x^3 + y^3 + xyz + yz^2 + xw^2 + w^3$	& $[0 : 1 : 2 : 1]$\\
				
				$x^3 + y^3 + xyz - yz^2 + xw^2 + w^3$	&   $[2 : 1 : 2 : 1]$\\
				
				$x^3 + y^3 + xyz + yz^2 + xw^2 - w^3$	& $[2 : 1 : 2 : 1]$\\ 
				
				$x^3 + y^3 + xyz - yz^2 + xw^2 - w^3$	& $[1 : 1 : 2 : 1 ]$\\  
				\hline 
			\end{tabular} 
		\end{center}
		
	\end{table}

	\begin{table}[h]  
		\caption{Isomorphism classes of N1 curves with $a=3$ over $\F_3$.}
		\label{table:repres_p3_N1}
		
		\begin{center}
			\begin{tabular}{cl}
				
				Representative $G$ & Type 1 points \\
				\hline 
				\hline
				$x^2y + y^3 + x^2z + y^2w + z^2w - w^3$	 &$[	1 : 2 : 1 : 1],[ 0 : 0 : 2 : 1 ]$\\
				$x^2y + y^3 + x^2z - y^2w + z^2w - w^3$	 &$[	0 : 2 : 0 : 1], [0 : 0 : 1 : 1 ]$\\
				$x^2y + y^3 + x^2z + z^3 + y^2w + z^2w + w^3$	& $[ 1 : 2 : 1 : 1],[ 1 : 1 : 2 : 1 ]$\\
				$x^2y + y^3 + x^2z + z^3 - y^2w + z^2w + w^3$	& $[0 : 0 : 1 : 1], [2 : 2 : 2 : 1] $\\
				$x^2y + y^3 + x^2z - z^3 + y^2w + z^2w + w^3$	& $[0 : 0 : 2 : 1],[ 1 : 0 : 1 : 0] $\\
				$x^2y + y^3 + x^2z - z^3 + y^2w + z^2w - w^3$	&  $[1 : 1 : 2 : 1],[ 1 : 0 : 1 : 0] $\\
				$x^2y + y^3 + x^2z - z^3 - y^2w + z^2w + w^3$	&  $[2 : 1 : 1 : 1], [1 : 0 : 1 : 0] $\\
				$x^2y + x^2z + y^2w + z^2w + w^3$&$[1 : 2 : 1 : 1],[ 1 : 1 : 2 : 1] $\\
				$x^2y + x^2z + y^2w + z^2w - w^3$&$[ 0 : 2 : 0 : 1], [0 : 0 : 2 : 1 ]$\\
				$x^2y + x^2z - y^2w + z^2w + w^3$& $ [0 : 2 : 0 : 1], [2 : 1 : 1 : 1] $\\
				
				\hline 
			\end{tabular} 
		\end{center}
		
	\end{table}

	\begin{table}[h]  
		\caption{Isomorphism classes of N2 curves with $a=3$ over $\F_3$, with $\beta^2 + 2\beta+ 2=0$.}
		\label{table:repres_p3_N2}

		\begin{center}
			\begin{tabular}{cl}
				
				Representative $G$ & Type 1 points \\ 
				\hline 
				\hline
				$ x^2y + y^3 + x^2z - z^3 - yzw - w^3$&$[\beta^5 : 1 : \beta^3 : 1], [\beta^7 : 1 : \beta : 1]$ \\
				
				$x^2y + x^2z - yzw - w^3$&$ [0 : \beta^5 : \beta^7 : 1], [0 : \beta^7 : \beta^5 : 1] $\\
				
				$x^2y + y^3 + x^2z - z^3 - yzw + w^3$&$ [\beta^2 : 0 : \beta : 1], [\beta^6 : 0 : \beta^3 : 1] $\\
				
				$x^2y + x^2z - z^3 - yzw + w^3$&$ [\beta^2 : 0 : \beta : 1], [\beta^6 : 0 : \beta^3 : 1] $\\
				$x^2y +y^3 + x^2z + z^3 - yzw + w^3$&$ [1 : \beta^2 : \beta^6 : 1], [1 : \beta^6 : \beta^2 : 1] $\\
				$x^2y - y^3 + x^2z - z^3 - yzw - w^3$&$ [\beta^7 : \beta^2 : 1 : 0], [\beta^5 : \beta^6 : 1 : 0] $\\
				$x^2y + x^2z + z^3 - y^2w - yzw + z^2w + w^3$&$ [2 : 0 : \beta^2 : 1], [2 : 0 : \beta^6 : 1] $\\
				$x^2y - y^3 + x^2z - y^2w - yzw + z^2w - w^3$&$ [\beta^7 : \beta^5 : 1 : 1], [\beta^5 : \beta^7 : 1 : 1] $\\
				$x^2y - x^2z - y^2w + yzw + z^2w + w^3$&$ [\beta^2 : \beta^3 : \beta^7 : 1], [\beta^6 : \beta : \beta^5 : 1] $\\
				$x^2y + x^2z - z^3 - y^2w - yzw + z^2w + w^3$&$ [\beta^2 : \beta^3 : \beta^7 : 1], [\beta^6 : \beta : \beta^5 : 1] $\\
				$x^2y - x^2z - z^3 - y^2w + yzw + z^2w - w^3$&$ [0 : 2 : \beta^2 : 1], [0 : 2 : \beta^6 : 1] $\\
				
				\hline 
			\end{tabular} 
		\end{center}
		
	\end{table}

	\subsection{Case $p=5$}\label{sec:exhaustive_p5}

	We  found the set of Cartier points on all curves in standard form over $\F_5$ with $a$-number $3$. There are some behaviors that are different from the case $p=3$. For instance, there are isomorphism classes of curves that attain the bound on the number of Cartier points, both with $f=1$ and $f=0$. It also happens that some of them have Type $2$ Cartier points (see Table~\ref{table:p=5GoodCurves_number_of_CPs} and Proposition~\ref{prop:bound_onCPs_p5}).  
	Another feature is that every degree distribution occurs when there are six distinct Type 1 points.

	\begin{table}[h]  
		\caption{Number of Type $1$ and $2$ Cartier points over $\F_5$.}
		\label{table:p=5GoodCurves_number_of_CPs}
		\begin{center}
			\begin{tabular}{c|cccccc|c}
				
				Case	& \multicolumn{6}{c|}{Curves with $n$ Type $1$ points} & Curves with Type $2$  \\ 
				
				& $1$ &2  &3  &4  & 5 & 6 &1  \\ 
				\hline 	\hline 
				D	& 14 & 0 &26  & 0 &0  &  19& 0 \\ 
				
				N1i	& 3 & 7 & 0 & 0 &0  &50  & 4 \\ 
				
				N1ii	& 0 &  1& 0 & 0 &0  &5  &1  \\ 
				
				N2	&2  & $1$ &0  & 0 &  0&6 &0  \\ 
				\hline 
				Total& 19&9&26&0&0&80&5\\ \hline
			\end{tabular} 
		\end{center}
	\end{table}

	\begin{example}\label{ex:T2CPs_p5}
		There are only $5$ isomorphism classes of curves in standard form with $a$-number $3$ over $\F_5$ that have a Type $2$ Cartier point. Here we show the cubic polynomial of each curve, together with the Type $2$ point. All of these curves have only one Type $1$ Cartier point, which implies there are no $p$-rank $1$ curves over $\F_5$ that reach the bound of seven total Cartier points.
		\begin{itemize}
			\item $x^2y + y^3 + x^2z + 2xyz + 2yz^2 + z^3 + 2y^2w + 2yzw + yw^2 + w^3$, $[1:0:3:0]$.
			\item $2x^2y + y^3 + x^2z + xyz - yz^2 - z^3 + 2yzw - 2yw^2 + zw^2 - w^3 $,  $[0:0:1:2]$.
			\item 	$2x^2y + y^3 + 2x^2z + xyz + y^2z - 2yz^2 - 2z^3 + 2y^2w - yzw + yw^2 - 2zw^2 - 2w^3 $, $[1:3:2:4]$. 
			\item 	$ x^2y + y^3 - 2x^2z + y^2z - yz^2 - y^2w - yzw + z^2w - yw^2 - 2zw^2 - 2w^3$, $[1:2:0:0]$.
			\item	$x^2y + 2x^2z - 2xyz + y^2z + yz^2 - 2y^2w + 2yzw + z^2w + yw^2 - 2zw^2 - 2w^3$,  $[1:4:1:1]$. 
		\end{itemize}
	\end{example}

	\begin{prop}\label{prop:bound_onCPs_p5}
		There are no curves in standard form over $\F_5$ with $p$-rank $1$ and $a$-number $3$ that reach Baker's bound of seven Cartier points.
		
		There are, up to $\F_5$-isomorphism, $17$ curves in standard form with $p$-rank $0$ and $a$-number $1$ that reach the bound of six Type $1$ Cartier points.
		
	\end{prop}
	\begin{proof}
		
		If a genus $4$ curve with $a=3$ has $p$-rank $0$ then by Proposition~\ref{prop:bound_cartier_points}, it has at most six Cartier points. If the $p$-rank is $1$, then the curve has at most seven Cartier points: six of Type $1$ and one of Type $2$. We use Algorithms~\ref{alg:Type1_points} and~\ref{alg:Type2_points} to compute the Cartier points on all N1, N2 and D curves with $a=3$ over $\F_5$ and find that only seventeen curves with $p$-rank $0$ have a total of six Cartier points. Example~\ref{ex:T2CPs_p5} shows that the only curves with Type $2$ Cartier points have fewer than six Type $1$ points, hence the total of Cartier points is less than the upper bound for all the $p$-rank $1$ curves. 
	\end{proof}

	\subsection{Case $p=7$}

	As mentioned in Section~\ref{sec:goodcurvesp=7}, for $p=7$ we also computed all the $a$-number $3$ curves of type D and of type N1i with $(b_1,b_2)=(0,0)$. Table~\ref{table:p=7GoodCurves_number_of_CPs} shows the number of Cartier points on these curves.

	\begin{table}[h]  
		\caption{Number of Cartier points on known D and N1 genus $4$ curves over $\F_7$.}
		\label{table:p=7GoodCurves_number_of_CPs}
		\begin{center}
			\begin{tabular}{c|cccccc|c}
				
				Case	& \multicolumn{6}{c|}{\#Curves with $n$ Type $1$ points} & \#Curves with Type $2$  \\ 
				
				& $1$ &2  &3  &4  & 5 & 6 &1  \\ 
				\hline 	\hline 
				D	& 0& 0 &3  & 9 &5  &  48& 0 \\ 
				
				N1i'& 0 & $1$ &0 & 3 &2  &23  & $1$ \\ 
				
				\hline 
				Total&0&1&3&12&7&71&1\\
				\hline
			\end{tabular} 
		\end{center}
	\end{table}

	\section{Examples}\label{sec:examples}
	
	In this section we show five  examples where we explicitly compute all the Cartier points of curves (or families of curves) with genus $4$. 
	\begin{example}\label{ex:CPbound_reached_p7_2}
		%From section {subsub:CPsF7}
		Let $X=V(F,G)$ be a genus $4$ curve over $\F_7$ where 
		\begin{align*}
		F&=2yz+2xw,\\
		G&=2x^2y + y^3 + x^2z + y^2z + 3z^3 + 2yzw + z^2w + 4yw^2 + 6zw^2 + 4w^3.
		\end{align*}
		Notice that $X$ belongs to the N1 case from Definition~\ref{def:standardform}. We will see how this curves attains Baker's bound on Cartier points. First, the Hasse--Witt matrix of $X$ is 
		\begin{align*}
		\begin{bmatrix}
		2&3&6&1\\
		0&0&0&0\\
		0&0&0&0\\
		0&0&0&0\\
		\end{bmatrix}.
		\end{align*}
		Thus the $a$-number is $3$ and the $p$-rank $1$, in which case the total bound on the number of Cartier point is $7$. By solving $H\bv=c\bv$ we see that the Type $2$ point is $[1:0:0:0]$.
		
		Also, the Type $1$ points are those in the intersection of $X$ and the hyperplane $2x+3y+6z+w=0$. There are two such points defined over $\F_7$, they are $[0 : 0 : 1 : 1]$ and $[2 : 4 : 3 : 1]$. The other four points are two pairs of $\sigma$-conjugate defined over $\F_{49}$ which are 
		$[6 : \alpha^{26} : \alpha^{22} : 1], [6 :  \alpha^{38} : \alpha^{10} : 1]$ and 
		$[4 :  \alpha^{33 }: \alpha^{23} : 1], [4 :  \alpha^{39 }:  \alpha^{17 }: 1]$, for $\alpha$ such that $\alpha^2 + 6\alpha + 3=0$.
	\end{example}

	\begin{example}\label{ex:p7_6T1_points_2}
		
		Let $X=V(F,G)$ be a genus $4$ curve over $\F_7$ where 
		\begin{align*}
		F&=	y^2 + 4z^2 + 2xw.\\
		G&=	3x^2y + 4xy^2 + 3y^3 + 5y^2z + 2xz^2 + yz^2 + 5z^3 + y^2w + 6yzw + 5zw^2 + 2w^3.
		\end{align*}
		The Hasse--Witt matrix of $X$ is 
		\begin{align*}
		H=\begin{bmatrix}
		1 &4 &1 &6\\
		5& 6& 5& 2\\
		0& 0& 0& 0\\
		1& 4& 1& 6\\
		\end{bmatrix}.
		\end{align*}
		In this case, $X$ has $p$-rank $1$, but the only eigenvector of $H$, up to scaling is $(1, 5, 0, 1)^T$, and this does not give a  point on $X$. On the other hand, there are four Type $1$ points: $[0 : \alpha^{14} : \alpha^{34 }: 1]$,$ [\alpha^7 : \alpha^{36 }: \alpha^{34} : 1]$, $[0 : \alpha^2 : \alpha^{46} : 1]$,	$[\alpha : \alpha^{12} : \alpha^{46} : 1]$, where $\alpha^2 + 6\alpha + 3=0$. The first two have multiplicity one and the others have multiplicity 2. 
	\end{example}

	\begin{example}\label{ex:CP_on_p11_curve}
		%%From section \label{subsub:CPsF11}
		Let $X=V(F,G)$ be a genus $4$ curve over $\F_{11}$ where 
		\begin{align*}
		F=&z^2 + 2yw,\\
		G=&9x^3 + xy^2 + 4y^3 + 9xyz + 2xz^2 + 8z^3 + 7xzw + 8xw^2 + zw^2 + 3w^3.
		\end{align*}
		The Hasse--Witt matrix of $X$ is 
		\begin{align*}
		H=\begin{bmatrix}
		1 & 5&  3& 10\\
		7 & 2& 10&  4\\
		2 &10 & 6 & 9\\
		4 & 9 & 1&  7
		\end{bmatrix}.  
		\end{align*}
		
		This curve has $a$-number $3$ and $p$-rank $1$. There are no Type $2$ points. Indeed, the only solution, up to scalar multiplication of $H\bv=c\bv$ is $\bv=(1, 7, 2, 4)^T$, and this does not give a point over $X$. 
		
		The Type $1$ points are those on the intersection of $X$ and the hyperplane with equation $x+5y+3z+10w=0$, that is $  [6: 9: 9: 1]$,
		$[ 7  +2\alpha  +7\alpha^2  +9\alpha^3 +10\alpha^4: 6  +4\alpha  +5\alpha^2  +4\alpha^3 +1\alpha^4:10   +4\alpha^2  +5\alpha^3 +6\alpha^4:1]$ with $\alpha^5+10\alpha^2+9=0$, and its four conjugate points.

	\end{example}

	\begin{example} \label{ex:GCd5-301_a4=0}
		
		Let $X$ be the genus $4$ curve over $\F_{5}$ embedded in $\PP^3$  as the zero locus of 
		\begin{equation} \label{eq:example_GCd5-301_a4=0}
		F=2yw+z^2, \text{       }
		G=a x^3+xw^2 +b y^3+c w^3+zw^2,
		\end{equation}
		with $a, b \in\F_{5}^\times$ and $c \in \F_{5}$. 
		By Proposition~\ref{prop:computation_hw}  the Hasse-Witt matrix of $X$ is
		\begin{equation}\label{eq:HWmatrix_ex_GCd5-301_a4=0}
		H= \begin{bmatrix}
		0& 0& 0 &0\\
		0&3ab^2& 4ab & 4a\\
		0& 0 & 0 & 0\\
		0& 0& 0& 0\\	 
		\end{bmatrix}.
		\end{equation}
		If $H\bv=c\bv$ for some $c\ne 0$, then $\bv$ must be in the subspace spanned by $(0,1,0,0)^T$. But $[0:1:0:0]$ is not a point on $X$. Therefore, $X$ has no Type $2$ point.
		
		Now, suppose $P=[x_0,y_0,z_0,w_0]$ is a Type $1$ Cartier point of $X$, and let $\bv=\bv_P$. Since $P$ is also on the hyperplane generated by the kernel of $H$, then $3b^2y_0+4bz_0+4w_0=0$, from where $w_0=3b^2y_0+4bz_0$. By evaluating $F$ at $P$ we see that 
		\begin{equation*}
		b^2y^2+3byz+z^2=0,
		\end{equation*}
		hence $z_0=2by_0$. Note that $y_0\ne 0$, because otherwise $P=[1:0:0:0]$, and this is not a point on $X$. So we can assume that $y_0=1$, and then $P=[x_0:1:2b:b^2]$. Evaluating $P$ at $G$ we get that
		\begin{equation*}
		ax_0^3+x_0+3b+cb=0.
		\end{equation*}
		
		Therefore, $X$ has $1, 2$ or $3$ Type $1$ Cartier points (counting without multiplicity), one for each of the roots of $ax^3+x+3b+cb$, that can be defined over $\F_5$, $\F_{5^2}$ or $\F_{5^3}$.

	\end{example}

	\begin{example} \label{ex:GCd5-301_a4}
		
		Let $X$ be the genus $4$ curve over ${\F}_{5}$ embedded in $\PP^3$  as the zero locus of 
		\begin{equation} \label{eq:example_GCd5-301_a4}
		F=2yw+z^2, \text{       }
		G=3a^2 x^3+xw^2 +bxyz+b^2a y^3+a w^3+zw^2,
		\end{equation}
		with $a, b \in\F_{5}^\times$. The curve $X$ is smooth and irreducible if $a\ne b$. By Proposition~\ref{prop:computation_hw}  the Hasse-Witt matrix of $X$ is
		\begin{equation}\label{eq:HWmatrix_ex_GCd5-301_a4}
		H= \begin{bmatrix}
		0& 0& 4ab^3 &4b\\
		0& 0 & 2a^3b^2  &2a^2 \\
		0& 0 & 0 & 0\\
		0& 0& 4b^3& 4a^3b\\	 
		\end{bmatrix}.
		\end{equation}
		So $X$ has $a$-number and $p$-rank $1$. There are no points on $X$ such that the corresponding vector $\bv$ is a solution to $H\bv^{(5)}=c\bv$ for $c$ in ${\F}_5^{\times}$. 
		
		Now, suppose that $P=[x_0,y_0,z_0,w_0]$ is a Type $1$ point. Then $P$ is in the hyperplane given by $ab^2z+w=0$, hence $w_0=4ab^2z$. If $z_0=0$, then $w_0=0$ and since $G(P)=0$, we have $3a^2x_0^3+ab^2y_0^3=0$. Clearly $x_0$ and $y_0$ cannot be zero, so let $y_0=1$ and then $P=[x_0:1:0:0]$ for $x_0$ a root of $3ax^3+b^2=0$. This gives a Type 1 point over $\F_5$ and two over $\F_{25}$. 
		
		On the other hand, if we assume $z_0=1$ then $w_0=4ab^2$ and by substituting $P$ in $F=0$ we obtain $y_0=3a^3b^2$. Therefore $P=[x_0:3a^3b^2:1:4ab^2]$, with $x_0$ a root of $3a^2x^3+x(a^2+3a^3b^3)+3a^2+4b^2$. For all possible $(a,b)$, this polynomial has three distinct roots, either over $\F_5$, $\F_{25}$ or $\F_{125}$.

	\end{example}

	%%%%%%%%%%%%%%%%%%%%%%%%%%%%%%%%%%%%%%%%%%%%%%%%%%%%%%%%%%%%%%%%
	%%%%%%%%%%%%%%%% Bibliography: bibtex of thebibliography
	%%%%%%%%%%%%%%%%%%%%%%%%%%%%%%%%%%%%%%%%%%%%%%%%%%%%%%%%%%%%%%%%

	%%%%%%%% Bibtex:
	\bibliographystyle{apa} % apa style bibliography
	%\bibliographystyle{} % math and engineering style bibliography
	%\bibstyle{plain}
	\bibliography{Cartier_points_article}% change sample to the name of your .tex file, e.g., thesis

\end{document}